\documentclass[11pt]{article}
	
\makeatletter
\renewcommand\section{\@startsection {section}{1}{\z@}%
                                   {-3.5ex \@plus -1ex \@minus -.2ex}%
                                   {2.3ex \@plus.2ex}%
                                   {\normalfont\fontfamily{phv}\fontsize{16}{19}\bfseries}}
\renewcommand\subsection{\@startsection{subsection}{2}{\z@}%
                                     {-3.25ex\@plus -1ex \@minus -.2ex}%
                                     {1.5ex \@plus .2ex}%
                                     {\normalfont\fontfamily{phv}\fontsize{14}{17}\bfseries}}
\renewcommand\subsubsection{\@startsection{subsubsection}{3}{\z@}%
                                    {-3.25ex\@plus -1ex \@minus -.2ex}%
                                     {1.5ex \@plus .2ex}%
                                     {\normalfont\normalsize\fontfamily{phv}\fontsize{14}{17}\selectfont}}
\makeatother

\usepackage{graphicx}
\usepackage{enumerate}
\usepackage{url} 
\usepackage{xcolor}
\usepackage{colortbl}

\usepackage{lineno,hyperref}
\usepackage{titlesec}
\usepackage{afterpage}
\usepackage{subcaption}
\usepackage[ansinew]{inputenc}
\usepackage[T1]{fontenc}
\usepackage{lmodern}
\usepackage[english]{babel}
\usepackage{setspace}
\usepackage{array}
\usepackage{amssymb,amsthm,amsmath}
\usepackage{calc}
\usepackage{cases}
\usepackage[square,numbers]{natbib}
\usepackage{color}
\usepackage{csquotes} 
\usepackage{booktabs}
\newsavebox\myboxgibt
\usepackage{multirow}
\usepackage{eurosym}
\usepackage{changes}
\usepackage{verbatim}
\usepackage{tabularx}
\usepackage[normalem]{ulem}
\usepackage{booktabs}
\usepackage{caption} 
\DeclareCaptionLabelFormat{cont}{#1~#2\alph{ContinuedFloat}}
\usepackage{rotating}
\usepackage{adjustbox}
\usepackage{optidef}
\usepackage{cleveref}
\usepackage{algorithm}
\usepackage[noend]{algpseudocode}

\usepackage{subcaption}
\usepackage{pdflscape}

\usepackage{pgfplots}
\pgfplotsset{compat = 1.15} 
\usetikzlibrary{pgfplots.statistics} 
\usepackage{pgfplotstable}
\usepgfplotslibrary{groupplots}
\usepackage{wrapfig}

\definecolor{color-multitsp}{rgb}{0.169, 0.514, 0.729}
\definecolor{color-singletsp}{rgb}{0.843, 0.098, 0.11}
\definecolor{color-efhasap}{rgb}{0.992, 0.682, 0.38}
\definecolor{color-efhs}{rgb}{0.671, 0.867, 0.643}
\definecolor{color-a0.25}{RGB}{170, 75, 214}
\definecolor{color-a0.5}{RGB}{77, 172, 189}
\definecolor{color-a1}{RGB}{173, 225, 142}
\definecolor{color-a2}{RGB}{255, 199, 104}
\definecolor{color-a4}{RGB}{176, 39, 41}

\newcommand{\argmin}[1]{{\operatorname{argmin}_{#1}}\;}
\usepackage{thmtools}
\usepackage{thm-restate}

\declaretheorem[name=Lemma]{lemma}

\declaretheorem[name=Definition]{definition}

\usepackage[resetlabels,labeled]{multibib}
\newcites{A}{Online Appendix References}
\bibliographystyleA{abbrvnat}

\theoremstyle{definition}

\newcommand{\ttour}{s^{\text{t}}}
\newcommand{\ttours}{S^{\text{t}}}

\newcommand{\dtour}{s^{\text{d}}}
\newcommand{\dtours}{S^{\text{d}}}
\newcommand{\depot}{v_0}
\newcommand{\policy}{\textsc{Alg}}
\newcommand{\ALG}{\textsc{Alg}}
\newcommand{\nrtrucks}{k^{\text{t}}}
\newcommand{\nrdrones}{k^{\text{d}}}
\newcommand{\nrvehicles}{k}
\newcommand{\OPThome}{\ensuremath{\text{OPT}}}
\newcommand{\OPT}{\ensuremath{\text{OPT}}}

\newcommand{\Ddrone}{D^{\text{d}}}
\newcommand{\segment}{\Delta}
\newcommand{\segmenttour}{t}
\newcommand{\multitsphome}{\ensuremath{\textsc{Optimistic}}}
\newcommand{\singletsphome}{\ensuremath{\textsc{Regretless}}}
\newcommand{\efhshome}{\ensuremath{\textsc{EFHS}}}
\newcommand{\efhahome}{\ensuremath{\textsc{EFHA}}}

\newcommand{\truckonly}{\ensuremath{\textsc{TruckOnly}}}

\newcommand{\tsp}{\textsc{TSP}}

\newcommand{\initial}{first}
\newcommand{\secondary}{second}

\newcommand{\tspmnnamelong}{the multiple Traveling Salesmen Problem with minmax objective}
\newcommand{\tspmnname}{MultiTSP}
\newcommand{\gstar}{\tilde{\mathcal{G}}}

\newcommand{\IRANDOM}{RANDOM}
\newcommand{\IGRAPHCLASS}{VAR}
\newcommand{\ISMALL}{SMALL}
\newcommand{\ILARGE}{LARGE}
\newcommand{\IBASE}{BASE}

\DeclarePairedDelimiter\set{\{}{\}}
\DeclarePairedDelimiter\abs{\lvert}{\rvert}

\newcommand{\bhktabelle}{B}

\newcommand{\N}{\mathbb{N}}
\newcommand{\NN}{\mathbb{N}}
\def\Oh{\ensuremath{\mathcal{O}}} 

\usepackage{footnote}
\usepackage{footmisc}
\renewcommand{\thefootnote}{\alph{footnote}}

\newcommand{\fastfootnote}[1]{%
\let\oldthefootnote=\thefootnote%
\setcounter{footnote}{0}%
\renewcommand{\thefootnote}{\fnsymbol{footnote}}%
\footnote{#1}%
\let\thefootnote=\oldthefootnote%
}

\begin{document}

\def\spacingset#1{\renewcommand{\baselinestretch}%
	{#1}\small\normalsize} \spacingset{1}

{
\title{The relief distribution problem with trucks and drones under incomplete demand information}

\date{}

\author{Aaron Neugebauer$^{a,b}$, Alena Otto$^c$\fastfootnote{Corresponding author, email: \texttt{alena.otto@tum.de}}, Marie Schmidt$^a$ \bigskip \\
{\footnotesize $^a$ University of W{\"u}rzburg, Institute of Computer Science,  Germany,} \\
{\footnotesize $^b$ Image Processing Department, Fraunhofer Institute for Industrial Mathematics ITWM, Germany,} \\
{\footnotesize $^c$ Technical University of Munich, Chair of Advanced Analytics in Manufacturing Management, Germany }
}
\maketitle



\begin{abstract}
	Disaster relief operations often take place under uncertainty regarding the extent of damage
	across locations. In this paper, we study the delivery of relief aid in the aftermath of disasters when delivery vehicles are assisted by surveillance drones and the demand for relief supplies is initially unknown.
	We introduce a stylized problem that arises in many emergency supply delivery settings -- the relief distribution problem (RDP).
	In RDP, emergency vehicles, referred to as trucks, must distribute relief supplies on a network, starting from the depot to potential delivery locations, whose demand is initially unknown. The trucks are assisted by surveillance drones, which cannot deliver relief supplies, but scout delivery locations to see whether relief supplies are needed or not. The objective is to visit all location by any vehicle, deliver supplies to all damaged ones, and minimizing the completion time of the relief operation.
	We study two natural policies for the online problem RDP which we evaluate in two ways: the competitive ratio quantifies the performance in comparison to an optimal solution obtained under full information on damages, the drone-impact is the ratio of the algorithm's performance to the best outcome achievable without drones.
	Through theoretical analysis and computational experiments, we characterize the operational trade-offs between these policies and derive insights for the effective deployment of drones in disaster response.
\end{abstract}

\noindent%

\spacingset{1.5} 

\section{Introduction} \label{sec:intro}
Disaster events are becoming more frequent and severe, especially due to global warming \citep{kundu2022emergency,kumar2024global,gao2025has}. In 2024 alone, 393 natural hazard-related disasters occurred, resulting in economic losses of USD 241.95 billion, 16,753 fatalities, and affecting 167.2 million people \citep{cred2025}.
Rapid and effective disaster response is therefore increasingly important. A central task in the aftermath of disasters is the distribution of emergency supplies, such as food and water, medication, blood, and other health care services, or other critical materials like, e.g., fuel or electric power generators to affected locations.

However, relief operations often take place under  uncertainty regarding the extent of damage across locations, implying that the corresponding demand for relief supplies is also initially unknown, as this information is revealed only gradually as the situation unfolds.
A key source of information is remote sensing imagery collected by satellites and drones, which enables rapid and automated assessment of disaster impacts -- such as building damage from floodings -- via AI-assisted image analysis \cite{polushko2025remote, hatic2025post}.

\begin{wrapfigure}{r}{0.41\textwidth}
	\begin{center}
		\includegraphics[width=0.40\textwidth]{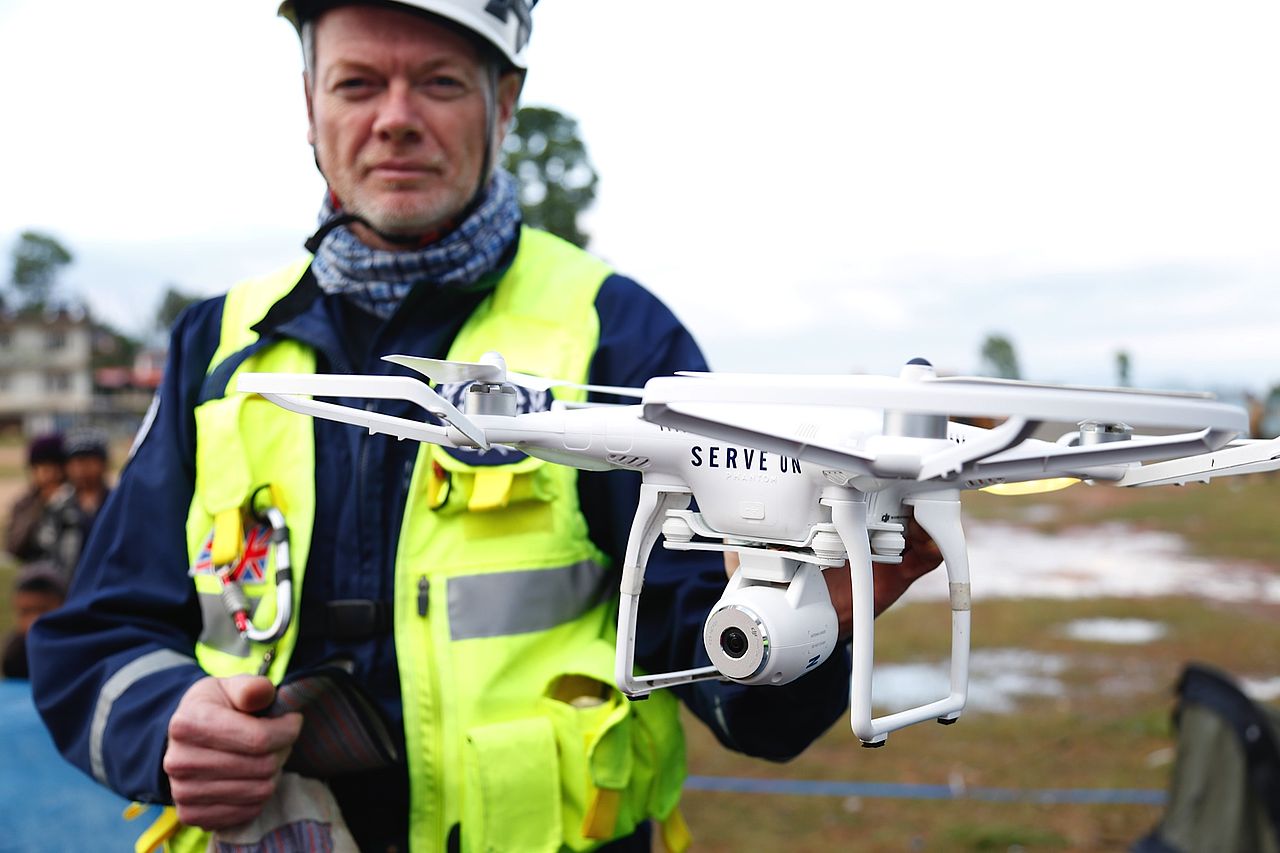}
	\end{center}
	\tiny{A photograph by Jessica Lea/DFID distributed under a CC BY 2.0 license}
	\caption{\footnotesize{Drone surveillance during earthquake in Nepal, 2015}}
	\label{fig:drone}
\end{wrapfigure}

Drones are frequently utilized for the collection of images, due to their ability to
deliver high-resolution images, operate in cloudy conditions, and their quick deployment capabilities \cite{gebrehiwot2019deep}.
In surveillance missions, drones typically carry lightweight sensors, allowing them to cover large geographic areas due to low energy consumption. By rapidly collecting situational information, surveillance drones can significantly improve operational decision-making during disaster response \citep{Alfaris2024689, Yucesoy2025545}. Figure~\ref{fig:drone} pictures the surveillance drone deployed in relief operations after the 7.8-magnitude earthquake in Nepal in 2015, which killed over 8000 people and injured thousands more \citep{govuk2026}.

\begin{figure}[t]
	\centering
	\includegraphics[scale=0.6]{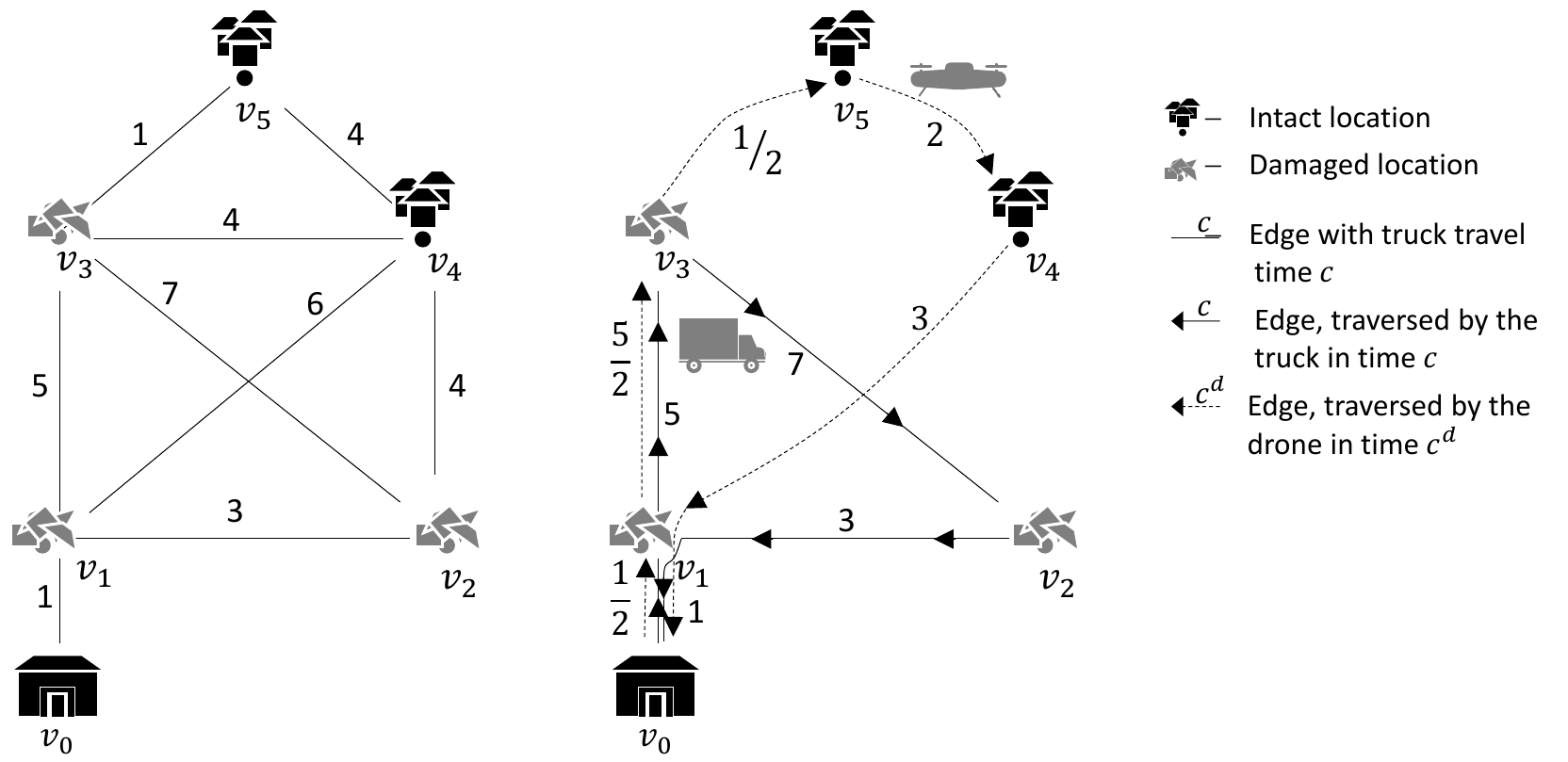}
	\caption{\footnotesize{Example  of an RDP$^*$ instance (full-information counterpart) \\
			\textit{Note.} \textit{Left figure}: An instance with $\alpha=2$, i.e., the drone is two times faster than the truck; damaged nodes $D=\{v_1,v_2,v_3\}$, 1 truck, and 1 drone. \\
			\textit{Right figure}: An optimal solution for this instance with makespan 17: the truck travels $(\depot, v_1, v_3, v_2, v_1, \depot)$, the drone travels $(\depot, v_1, v_3, v_5, v_4, v_1, \depot)$.
		} }
	\label{fig:1}
\end{figure}

In this paper, we study the delivery of relief aid in the aftermath of disasters when delivery vehicles are assisted by surveillance drones.
We introduce a stylized problem that arises in many emergency supply delivery settings -- \textit{the relief distribution problem (RDP)}. In RDP (see Figure~\ref{fig:1}), $\nrtrucks\in \mathbb{N}$ emergency vehicles, referred to as \textit{trucks}, must distribute relief supplies starting from the depot $\depot$ to a nonempty set $C$ of potential delivery locations -- \textit{villages} -- and return to the depot $\depot$ via a road network, represented as an undirected graph. We exclude trivial instances by assuming that all nodes in $C$ are at nonzero distances from $\depot$. The trucks are assisted by $\nrdrones\in \mathbb{N}$ surveillance drones, which travel $\alpha>0$ times faster than trucks. If the villages remained intact, they do not require relief aid. If they experienced some destruction, then they require relief aid.
Only a truck can deliver relief aid, the drones cannot do it because of capacity restrictions. We call villages that experienced destruction and require relief aid as \textit{damaged} and denote them as set $D\subseteq C$. Crucially, set $D$ is \textit{initially unknown}. The status of each village -- damaged or not -- is revealed only after it is visited by either a truck or a drone. The objective is to visit all locations in $C$, deliver supplies to all damaged villages in $D$ using trucks, and minimize the completion time of the relief operation, defined as the time when the last vehicle returns to the depot. The RDP is an \textit{online} problem, meaning that information is only partially available and arrives dynamically \citep[cf.][]{albers2003}.

We evaluate online algorithms for RDP along two key dimensions: \textit{performance} and \textit{the impact of drones}.
A \textit{performance} gap measure should provide information on the deviation of the algorithm's result from what can be achieved by any other, including the best possible algorithm. In the absence of distributional information on the uncertain parameters, competitive ratio emerges as the most widespread and well-formalized measure for algorithm's performance \citep{albers2003, fiatandwoeginer1998}. \textit{Competitive ratio} $\sigma(ALG)$ of online algorithm $\ALG$ is the worst-case ratio of the online algorithm's cost to the cost of an optimal offline algorithm, where all data are known a priori:
\begin{align*}
	 &  &  & \sigma(ALG)=\sup_{I\in \mathcal{I}}\frac{\ALG(I)}{\OPT(I^*)}, &
\end{align*}
where $I$ is an RDP instance from the set of all possible instances $\mathcal{I}$, $I^*$ is the respective instance with full information, $\ALG(I)$ is the result of $\ALG$ on $I$, and $OPT(I^*)$ is the optimal objective value of $I^*$.
The algorithm's competitive ratio provides information that simulations cannot: it offers a \textit{guaranteed performance bound} across all possible scenarios. When delivery times in a disaster relief operation are long, the competitive ratio indicates the extent to which delays are caused by the disaster itself versus algorithmic limitations. For example, if the competitive ratio equals one, no algorithm can improve the outcome, and any extended delivery duration reflects the inherent difficulty of the instance rather than shortcomings of the algorithm.

Similarly, \textit{the impact of using drones} in an online RDP algorithm $\ALG$ is measured by the worst-case ($\bar{\omega}$) or best-case ($\underline{\omega}$) \emph{drone-impact ratio}, i.e., the ratio of the algorithm's performance to the best outcome achievable without drones. In the absence of drones, the problem reduces to \textit{\tspmnnamelong\ (\tspmnname)} with $\nrtrucks$ trucks visiting all nodes in $C$ \citep[cf.][]{bektas2006, francaetal1995}. We introduce this problem formally later and denote its optimal objective value by $\tsp_{\nrtrucks,0}(C(I))$. Then we can define the best-case and the worst-case \emph{drone-impact ratio} as, respectively:
\begin{align*}
	 &  &  & \bar{\omega}(ALG)=\sup_{I\in \mathcal{I}}\frac{\ALG(I)}{\tsp_{\nrtrucks, 0}(C(I))},      &
	 &  &  & \underline{\omega}(ALG)=\inf_{I\in \mathcal{I}}\frac{\ALG(I)}{\tsp_{\nrtrucks,0}(C(I))}. &
\end{align*}

Our main finding is that an online policy $\ALG$ that is optimal in terms of competitive ratio may not be \textit{regretless}, even if the drones are fast ($\alpha>1$); that is, it may achieve substantially worse performance compared with truck-only solutions ($\bar{\omega}(ALG)>1$). This result is somewhat counterintuitive: although drones provide additional surveillance capabilities, their introduction can actually significantly increase the completion time of relief operations in some cases. In realistic instances, we observe makespans that are up to 100\% larger after introducing surveillance drones. This phenomenon arises from the dynamic arrival of information, which creates a nontrivial trade-off between exploration and delivery efficiency.

Motivated by this observation, we study two natural policies for RDP: $\multitsphome$, which achieves the best possible competitive ratio, and $\singletsphome$, which is regretless and achieves the best possible  drone-impact ratios $\bar{\omega}$ and $\underline{\omega}$. Through theoretical analysis and computational experiments, we characterize the operational trade-offs between these policies and derive insights for the effective deployment of drones in disaster response.

RDP differs from classical online routing problems in several important ways (see Section~\ref{sec:lit_review}). First, it is \textit{semi-online} and involves \textit{exploration}. Semi-online means that some information -- in RDP, the set of potential locations -- is known \textit{a priori}. Exploration means that routing decisions influence the timing at which new information is revealed. Second, the problem involves heterogeneous vehicles whose operations have to be synchronized: drones that perform surveillance-only tasks and vehicles that perform deliveries; the information collected by surveillance drones becomes immediately available for trucks, who may change their routes.

Our contributions are threefold:
\begin{itemize}
	\item To the best of our knowledge, we are the first to formulate RDP.

	\item We design two policies: $\multitsphome$, which achieves the best possible competitive ratio in a range of practice-relevant settings, and $\singletsphome$, which achieves the best possible  drone-impact ratios $\bar{\omega}$ and $\underline{\omega}$. We \textit{systematically} analyze these policies with respect to the competitive ratio and  drone-impact ratios
	      as a function of the drone speed
	      $\alpha$, the number of trucks $\nrtrucks$, and the number of drones $\nrdrones$.

	\item Through extensive computational experiments -- including simulated mountain and coastal environments, which are most prominent in disaster relief operations -- we show that theoretically derived bounds \textit{closely approximate the observed performance} of the policies.  We
	      benchmark these policies and derive managerial insights for the effective deployment of surveillance drones in relief logistics.
\end{itemize}

We proceed with a literature review (Section~\ref{sec:lit_review}) and the formal problem statement (Section~\ref{sec:basic_problem}).  Section~\ref{sec:overview} summarizes the
main results of this article: the competitive ratio and  drone-impact ratio results for $\multitsphome$ and $\singletsphome$. This section can be read independently of the technical details of the proofs of Section~\ref{sec:analysis}. Section~\ref{sec:experiments} reports computational experiments and we conclude with
a discussion and future research directions in Section~\ref{sec:conclusion}.

\section{Literature review} \label{sec:lit_review}
\begin{table}
	\centering
	\scriptsize
	\begin{tabular}{p{1.2cm}p{2.1cm}p{2.8cm}p{4.1cm}p{1.3cm}p{0.3cm}p{0.3cm}}
		\toprule
		Article                                       & Application                   & Source of dynamism           & \# drones                      & \# trucks                   & \multicolumn{2}{c}{Analysis$^\dagger$}              \\
		                                              &                               &                              &                                &                             & CR                                     & DI         \\
		\midrule
		\citep{coleman2024optimaldeliveryfaultydrone} & communication \quad failure   & location of the failed drone & 2 (failed, recovery)           & --                          & \checkmark                             & --         \\
		\citep{deBerg2024254}                         & search                        & location of the target       & 1                              & --                          & \checkmark                             & --         \\
		\citep{Gharaibeh20253502}                     & location of charging stations & energy requests              & multiple                       &                             & \checkmark                             & --         \\
		\citep{Gou2023}$^{\dagger\dagger}$            & online deliveries             & requests                     & 2 (independent, truck-carried) & 1                           & \checkmark                             & --         \\
		\citep{Jana2024488}                           & online deliveries             & requests                     & multiple                       & 1$^{\dagger\dagger\dagger}$ & \checkmark                             & --         \\
		\citep{Otto20251198}                          & relief deliveries             & arc damage                   & 1                              & 1                           & \checkmark                             & --         \\
		this study                                    & relief deliveries             & relief demand                & multiple                       & multiple                    & \checkmark                             & \checkmark \\
		\bottomrule
		\multicolumn{7}{l}{\scriptsize{$^\dagger$ CR - competitive ratio analysis, DI - theoretical analysis of the drone's impact: worst-case and best-case bounds }}                                                                    \\
		\multicolumn{7}{l}{\scriptsize{\qquad on the ratio to truck-only solution. \quad $^{\dagger\dagger}$ - precise classification is unclear due to inconsistencies in the paper.}}                                                   \\
		\multicolumn{7}{l}{\scriptsize{$^{\dagger\dagger\dagger}$ - the truck travels on the line.}}                                                                                                                                      \\
	\end{tabular}
	\caption{Studies on competitive ratio analysis for drone routing applications}
	\label{tab:litrev_drones_online}
\end{table}

Research on drone operations and humanitarian logistics, particularly the distribution of emergency supplies following disasters, has expanded significantly in recent years \citep[e.g.,][]{Adsanver2024384, Chung2020, kundu2022emergency, optimization-approaches-for-civil-applications-of-uav-survey}.
Among these applications, the delivery of emergency supplies in disrupted environments has emerged as one of important operational uses \citep[e.g.,][]{Alfaris2024689, Yucesoy2025545}.

RDP is an online routing problem. Seminal works by \citet{Ascheuer2000} and \citet{Ausiello1999} initiated the competitive ratio analysis of online routing problems in which information about  \textit{requests} (nodes to be visited) arrives dynamically over time. Unlike in  RDP, the locations of requests in these models are not known a priori.
Since then, routing problems with dynamically arriving service requests have been studied in a variety of settings \citep[e.g.,][]{ausielloetal2008,jailletandwagner2006,jailletandwagner2008}.
RDP differs from classical online routing models, because it is semi-online, involves exploration, and heterogeneous vehicles.
To the best of our knowledge, only \citet{Bampis202365} consider an online routing problem with known request locations, but they do not consider exploration.
A related line of work studies \textit{the online exploration problem}, where vehicles explore a graph by visiting its nodes. In these models, however, only the source node is initially known; nodes and edges are revealed upon visitation, and limited information about outgoing edges (e.g., outdegree) becomes available when a node is reached \citep{birxetal2021,Megow201262}. More recently, online problems have also been studied in the presence of \textit{advice} and within the framework of \textit{learning-augmented algorithms}, which incorporate predictions together with measures of prediction error \citep{Antoniadis2026434,Shin2025}.
Within this perspective, RDP can be viewed as a variant of the graph exploration problem with surveillance drones and trucks, in which the underlying graph structure is perfectly predicted a priori, while only the status of nodes (damaged or not) must be discovered through exploration. To the best of our knowledge, such a variant has not been studied in the literature.

Overall, the literature applying competitive ratio analysis to drone-assisted routing problems is still nascent (see Table~\ref{tab:litrev_drones_online}). The settings studied thus differ substantially from ours and typically involve at most a single truck. While some works consider trucks for delivery operations \citep{Gou2023,Jana2024488,Otto20251198}, only \citet{Otto20251198} incorporates the exploration of network damage. This study, which is closest to our setting, examines emergency supply delivery under uncertainty about road conditions, where vehicles learn the status of roads only upon traversing them.

Beyond competitive ratio analysis,
a few studies compare truck-only delivery systems with integrated truck-and-drone systems, often analyzing the potential efficiency improvements obtained through drone assistance \cite{opt-approaches-for-tsp-with-drones,poikonen2017vehicle}. However, these comparisons are typically conducted in deterministic settings where all customer locations are known in advance.

To conclude, to the best of our knowledge, this study is the first to analyze  RDP -- an online routing problem with given delivery locations, exploration, and heterogeneous vehicles -- and one of the few papers that apply competitive-ratio analysis to routing problems with drones.

\section{The relief distribution problem (RDP)} \label{sec:basic_problem}

Before we state RDP  in Section~\ref{sec:basic_problem_UD}, we formulate it \textit{full-information counterpart} in Section~\ref{sec:basic_problem_FI}, which we mark with $^*$ as RDP$^*$. Section~\ref{sec:assumptions} discusses assumptions.

\subsection{Relief distribution problem under full information } \label{sec:basic_problem_FI}

We consider a setting with $\nrtrucks \in \NN$ trucks and $\nrdrones \in \NN$ drones. The street network is described as an  undirected graph $G=\left(V, E, c \right)$, where $V$ is the set of nodes, $E\subseteq\{\{i,j\}| i,j\in V \}$ defines the set of edges, and $c$ are edge labels.
Nodes $V=\{\depot\}\cup C$ consists of a depot node $\depot$, where both trucks and drones start and need to return after completing their tours, and of a \textit{nonempty} set of potentially damaged nodes $C$, which have to be visited by either of the vehicles (truck or drone) at least once. The subset $D \subseteq C$ refers to \textit{damaged nodes}.
All nodes in $D$ have to be visited by a truck.

We assume that $G$ is connected but not necessarily complete.
Edge labels describe distances with $c(i,j) > 0, \forall \{i, j\} \in E$ that fulfill the triangle inequality, i.e., the graph is metric.
If $\set{i, j} \notin E$, we set $c(i, j)$ to the distance along a shortest path between $i$ and $j$.
Observe that the nonempty set $C$ and positive distances between distinct locations guarantee that denominators in the competitive or  drone-impact ratios are nonzero.
We assume homogeneous trucks and homogeneous drones and normalize the truck speed to be $1$ and the drone speed to be $\alpha>0$, i.e., $c(i,j)$ is the travel time needed for the truck to traverse $\{i,j\}\in E$, and by $c^d(i,j)=\frac{1}{\alpha}\cdot c(i,j)$ we denote the drone flight time for this edge.

Denote by a \emph{tour} (or \emph{path}) a sequence of adjacent edges, which can be specified, alternatively, by the sequence of vertices defining these edges.
In this papers we sometimes abbreviate the tour notation by omitting nodes from the defining sequence. To reconstruct a tour (i.e., an actual path in the network) from this abbreviated representation, take any shortest path between vertices that succeed each other in the defining sequence. We use $c(t)$ or $c(i_1,i_2,\ldots,i_k)$ to denote the length of tour $t=(i_1,i_2,\ldots,i_k)$.

The objective of the \textit{relief distribution problem under full information RDP$^*$}
is to find \textit{truck tours} $\{\ttour_i:i=1,\ldots,\nrtrucks\}$ and \textit{drone tours} in $\{\dtour_j:j=1,\ldots,\nrdrones\}$ in $G$ to minimize the \textit{makespan}, i.e., the duration of the longest tour,
such that: \textit{(i)}  All tours start and end at the depot $\depot$; \textit{(ii)}   Each node in $C$ is visited either by a truck or by a drone; \textit{(iii)} each damaged node $d\in D$ is visited by a truck.
Note that we do not require every vehicle to depart from the depot.

We call an instance of RDP$^*$ as $I^*=(\nrtrucks,\nrdrones,\alpha,G,D)$ and denote its optimal objective value as $\OPThome(I^*)$. Especially when graph $G$ of $I^*$ is not complete, vehicles may visit a node several times, cf. the example in Figure~\ref{fig:1}.
For $\nrdrones=0$ and $\nrtrucks=1$, RDP$^*$ corresponds to a traveling salesman problem (\tsp). NP hardness of our problems follows by reduction from this problem.

\subsection{Relief distribution problem under uncertain damage} \label{sec:basic_problem_UD}

Immediately after a disaster, the information on which nodes from the set $C$ are damaged may not be available. We study this setting in \textit{the relief distribution problem (RDP)}.
An \emph{instance} $I=(\nrtrucks,\nrdrones,\alpha,G,D)$ of RDP  contains the same data as its full information counterpart -- instance $I^*$ -- but the status of the nodes in set $C$ (i.e., whether they pertain to the set of damaged nodes $D$ or not) is not known in the beginning and is only \emph{discovered} when a vehicle has arrived at the respective node.

In problems under incomplete information, such as RDP, one does not speak of solutions, but of \textit{policies}. A policy specifies the actions of trucks and drones in each information state -- i.e., `moving to node $v_i$ from the current location along an edge in graph $G$' or `waiting' -- based only on information available initially and the damage status of the visited nodes. Luckily, because our analysis focuses on best-possible as well as particular given policies, we avoid a full specification of the state and action spaces, which would lead to introducing unnecessary cumbersome notation.

Each algorithm $\ALG$ we study is a policy that always yields a feasible solution for the full-information counterpart $I^*$ and the objective value is evaluated as defined in Section~\ref{sec:basic_problem_FI}.

\subsection{Discussion of assumptions} \label{sec:assumptions}
RDP represents a stylized problem that summarizes common characteristics of many relief distribution settings. We make several assumptions and omit the following details, whose analysis we defer to future research:

\textit{Immediate information transmission}. We assume that damage is recognized instantly when a node is visited by any vehicle, and information is transmitted instantly.

\textit{Immediate unloading and relief distribution}.
We do not consider service times in the nodes, i.e., relief aid is distributed instantly when the node is visited by a truck.

\textit{Unlimited truck capacity}. In line with the literature \citep{Gou2023, Jana2024488, Otto20251198}, we assume that trucks have sufficient capacity. This assumption is particularly relevant for emergency supplies such as medication, first-aid and hygiene kits, and electric power generators.

\textit{Constant truck-to-drone travel time ratio $\alpha$.}
In line with the literature \citep{opt-approaches-for-tsp-with-drones,poikonen2017vehicle}, we assume that drone travel times are proportional to truck travel times, i.e., $c(i,j) = \alpha\cdot c^d(i,j), \forall i,j \in V$.
While in practice truck speeds depend on road conditions, and drones may travel more directly but are affected by wind, this abstraction facilitates tractable analysis.
To apply our worst- and best-case results, practitioners may calibrate $\alpha$ to the corresponding expected worst (best) ratio $\frac{c(i,j)}{c^d(i,j)}$ over all relevant node pairs in $G$ for the operational setting under consideration.

\textit{Sufficient energy}. We assume that drone and truck tour lengths are not constrained by energy limitations. This assumption is not overly restrictive: for instance, lightweight drones, such as those developed by Quantum Systems, can achieve flight ranges of up to 100 km \citep{magdolna2025}, with further improvements expected as the technology matures.

\section{Overview of main results} \label{sec:overview}
\begin{figure}
	\centering
	\begin{subfigure}[b]{0.48\textwidth}
		\centering
		\includegraphics[scale=0.6]{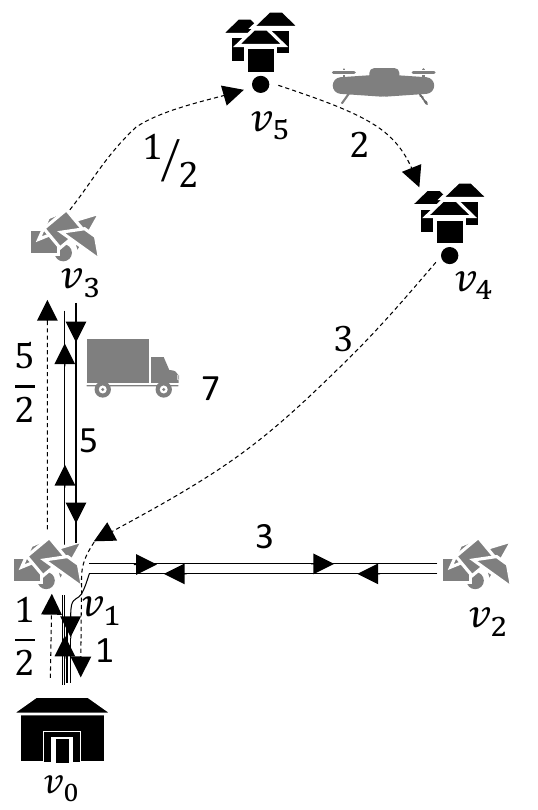}
		\caption{\footnotesize{\multitsphome\\ }}
		\label{fig:2multitsp}
	\end{subfigure}
	\begin{subfigure}[b]{0.48\textwidth}
		\centering
		\includegraphics[scale=0.6]{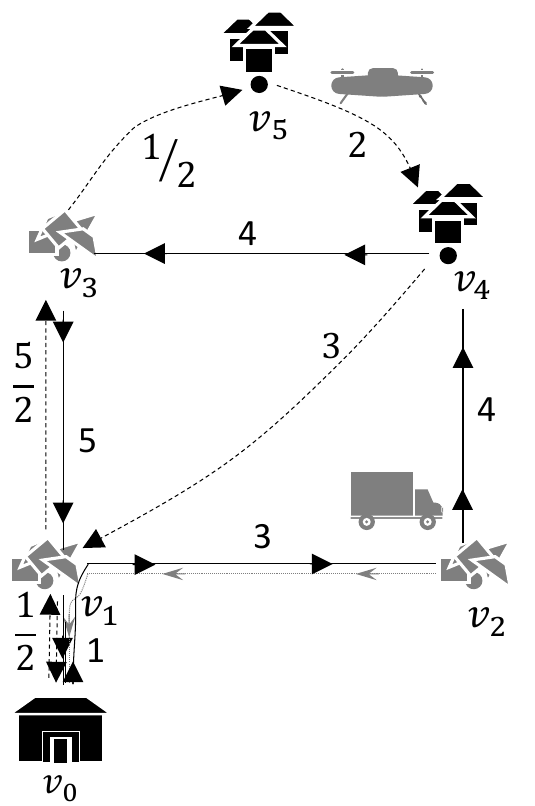}
		\caption{\footnotesize{\singletsphome\\}}
		\label{fig:2singletsp}
	\end{subfigure}
	\caption{\footnotesize{\multitsphome\ and \singletsphome\ for the example in Figure~\ref{fig:1} }}\label{fig:2}
\end{figure}

Before formally introducing $\multitsphome$ and $\singletsphome$ in Section~\ref{sec:policies}, this section provides intuition for these policies and an overview of our main results.

Policy $\multitsphome$ treats trucks and drones as co-equal vehicles. Initial tours are constructed for all vehicles so that every node in $C$ is visited and the depot is reached again in minimum time. After that, trucks visit the damaged nodes identified by the drones.

In contrast, $\singletsphome$ treats drones as purely supporting vehicles. It first constructs truck tours that visit all nodes in $C$, ignoring the presence of drones. When trucks start executing these tours, drones inspect the nodes toward the end of the tours, allowing trucks to skip nodes whose status has been revealed as not damaged.

Figure~\ref{fig:2} illustrates the two policies on the example from Figure~\ref{fig:1} for $\nrtrucks=\nrdrones=1$ and $\alpha=2$. Under \multitsphome, the truck first travels $(\depot, v_1, v_2, v_1, \depot)$ and the drone travels $(\depot, v_1, v_3, v_5, v_4, v_1, \depot)$. After both vehicles return to the depot at time $9$, the truck visits the remaining damaged nodes -- i.e., those not yet visited by the truck -- via $(\depot, v_1, v_3, v_1, \depot)$, resulting in a makespan of 21.
Under \singletsphome,
first, the tour $(\depot, v_1, v_2, v_4, v_5, v_3, v_1, \depot)$ is partitioned in a truck segment $(\depot, v_1, v_2)$ and a drone segment $(v_4, v_5, v_3, v_1, \depot)$.
Both vehicles traverse their segments - the drone in reverse order.
At time 4, the truck reaches the last node of its segment, $v_2$, while the drone is en route between $v_5$ and $v_4$.
At this point, the truck is reassigned to cover the unexplored node $v_4$ and the damaged node $v_3$ before returning to the depot via a shortest path  $(v_2, v_4, v_3, v_1, \depot)$, while the drone completes its route. The resulting makespan is 18.

\begin{table}
	\centering
	\footnotesize
	\begin{tabular}{llll}
		\toprule                                                                                                                                                                                                                                    \\
		Measure                     & Any deterministic online policy $\policy$                                                      & $\multitsphome$                                                                           & $\singletsphome$ \\
		\midrule
		Drone-impact ratio                                                                                                                                                                                                                          \\
		(best-case $\underline{\omega}$)
		                            & $\ge\frac{1}{1+\alpha\lceil\frac{\nrdrones}{\nrtrucks}\rceil} $
		                            & $\le \frac{1}{1+\alpha\frac{\nrdrones}{\nrtrucks}} $
		                            & $ \le   \frac{1}{1+\alpha\frac{\nrdrones}{\nrtrucks}} $ for some $\alpha, \nrtrucks,\nrdrones$
		\\
		(worst-case $\bar{\omega}$) &
		$\ge 1$
		                            & $\ge \min \{1+\frac{1}{\alpha},\frac{1+\alpha}{1+\varepsilon}\}$
		                            & $=1$
		\\
		                            &                                                                                                & $\le \min\{2,1+\frac{1}{\alpha}\left\lceil\frac{\nrtrucks}{\nrdrones}\right\rceil\}$                         
		\\
		\midrule
		Competitive ratio $\sigma$
		                            & $\ge\min\{2,1+\alpha\lceil\frac{\nrdrones}{\nrtrucks}\rceil\}$ for $\frac{1}{\alpha} \in \NN$
		                            & $= \min\{2,1+\alpha\lceil\frac{\nrdrones}{\nrtrucks}\rceil\}$
		                            & $\le 1+\alpha\lceil\frac{\nrdrones}{\nrtrucks}\rceil$
		\\
		                            & $\ge 2-\frac{1}{2\alpha}$ for $\alpha\in \NN$
		                            &                                                                                                & $=1+\alpha\lceil\frac{\nrdrones}{\nrtrucks}\rceil$ for some $\alpha, \nrtrucks,\nrdrones$                    \\
		                            & $\ge \min \{2\alpha,\frac{2}{\alpha}\}$                                                        &                                                                                           &                  \\
		\bottomrule
	\end{tabular}
	\caption{\footnotesize{Overview of main results}}
	\label{tab:overview}
\end{table}

Table~\ref{tab:overview} summarizes our main results on the competitive and drone-impact ratios of \multitsphome\ and \singletsphome. Both policies fully exploit the potential of drone technology, since they achieve the best-case drone-impact ratio $\underline{\omega}$ of $\frac{1}{1+\alpha\lceil\frac{\nrdrones}{\nrtrucks}\rceil}$, which is the \textit{best possible for any online algorithm} $\policy$ (proven for integer-valued ratios $\lceil\frac{\nrdrones}{\nrtrucks}\rceil$).  This formula highlights that the achievable benefit from using drones increases with both the drone speed $\alpha$ and the ratio of available drones to trucks $\frac{\nrdrones}{\nrtrucks}$. However, the policies exhibit markedly different relative performance with respect to the remaining ratios.

\singletsphome\ achieves the \textit{best possible} worst-case drone-impact ratio $\bar{\omega}$ of 1, meaning that using drones under this policy never worsens the makespan of relief supply deliveries compared to truck-only operations. In other words, under \singletsphome, we never regret using drones. However, this comes at a cost in terms of performance: the competitive ratio $\sigma$, which compares the policy's result to the optimal solution under full information, can be substantially higher for \singletsphome\ than for \multitsphome\ when drones are fast ($\alpha > 1$), e.g., at least 50\%  higher if $\alpha=2$ and at least 100\% higher if $\alpha=3$. One mechanism leading to performance loss is the following. Since \singletsphome\ has trucks traverse the same tours as in truck-only operations -- only skipping nodes identified as undamaged by the drones -- this policy is vulnerable when damages are concentrated along a few truck routes, preventing effective node skipping.

In contrast, \multitsphome\ achieves low competitive ratios $\sigma$ of $\min\{2,1+\alpha\lceil\frac{\nrdrones}{\nrtrucks}\rceil\}$, which are \textit{the best possible} if the truck speed is an integer multiple of the drone speed, meaning that no deterministic online policy can outperform it in this setting. However, under \multitsphome, the use of drones can sometimes increase the makespan of relief delivery compared to truck-only operations, even when drones are fast ($\alpha>1$).  When drones and trucks have the same speed, we can construct examples where the makespan doubles relative to a truck-only solution.

Our theoretical bounds and tightness results also hold for metric and planar graphs, as all constructed examples are based on such graphs (a graph is called \emph{planar} if it can be drawn on the plane so that its edges intersect only at their endpoints). This is particularly relevant in practice, since distances in street networks (graphs) are metric (i.e., they satisfy the triangle inequality), and these networks are often planar or nearly planar.

\section{Drone-impact and competitive ratio analysis}\label{sec:analysis}
After introducing some notation and basic lemmata that will be used throughout our analysis (Section~\ref{sec:notation}), we formally define $\multitsphome$ and $\singletsphome$ policies (Section~\ref{sec:policies}), present best-case (Section~\ref{sec-comparison_truck_bestcase}) and worst-case (Section~\ref{sec-comparison_truck_worstcase}) drone-impact ratio analysis as well as competitive ratio analysis results (Section~\ref{sec:competitive}).
In our analysis, we proceed as follows. First, we derive bounds on the ratios -- upper bounds for $\bar{\omega}(ALG)$ and  $\sigma(ALG)$, and a lower bound for $\underline{\omega}(ALG)$. We then present extreme instances -- worst-case for $\bar{\omega}(ALG)$ and $\sigma(ALG)$, and best-case for $\underline{\omega}(ALG)$ -- that establish the tightness of these bounds.

\subsection{Notation and basics}\label{sec:notation}

First, we introduce a related problem -- \textit{\tspmnnamelong\ (\tspmnname)} --  which is
essential for bounding the optimal objective value of RDP$^*$ and for analyzing policies in the subsequent sections.
After discussing some basic relations,  we define the \truckonly\ policy.
This simple policy provides the baseline for the drone-impact ratios.
Finally, we define specific graphs used in our extreme examples below.

We state \tspmnname\  as follows:
\begin{definition}[$\tsp_{m,n}(W)$, $S^*_{m,n}(W)$]
	\label{defi:tspmn}
	Given graph $G=(V,E,c)$ with a designated depot $\depot$, $C=V\setminus \{\depot\}$, positive edge labels $c$, a drone speed $\alpha$, nonnegative integers $m$ and $n$ with $m+n\geq 1$, and a node set $W\subseteq C$, \tspmnname\ is to
	find a set of $m+n$ tours $S^*=\ttours\cup \dtours$ in $G$ that start and end at the depot and visit all nodes in $W$,
	such that, if each tour in $\ttours=\{\ttour_i:i=1,\ldots,m\}$ is performed by a truck and each tour in  $\dtours=\{\dtour_j:j=1,\ldots,n\}$ is performed by a drone the overall makespan of these tours is minimized. We denote the respective \textit{optimal} makespan as $\tsp_{m,n}(W)$ and an optimal solution as $S^*_{m,n}(W)$.
\end{definition}

Note that in the notation introduced in Definition~\ref{defi:tspmn} and in the remainder of the paper, we simplify notation by omitting explicit dependence on the underlying instance or graph -- by writing, e.g., $W$ instead of $W(G)$ or $W(I)$ whenever the context is clear.

Lemma~\ref{lemma:basic} summarizes some basic relations for $\tsp_{m,n}(W)$:
\begin{restatable}{lemma}{lemmabasic}\label{lemma:basic}
	The following relations hold for all graphs $G=(V,E,c)$ with $V=C\cup\{\depot\}$, all drone speeds $\alpha>0$, all sets $W\subseteq C$ and all $m,n\ge 0, m+n\geq 1$.

	\noindent\textit{[Non-deterioration w.r.t. less nodes]}
	\begin{align}
		\tsp_{m,n}(W')\le \tsp_{m,n}(W) &  & \forall W'\subseteq W \label{eq:morenodes}
	\end{align}
	\textit{[Non-deterioration w.r.t. more vehicles]}
	\begin{align}
		\tsp_{m',n}(W)\le \tsp_{m,n}(W) &  & \forall m'\geq m \label{eq:moretrucks} \\
		\tsp_{m,n'}(W)\le \tsp_{m,n}(W) &  & \forall n'\geq n \label{eq:moredrones}
	\end{align}
	\textit{[Reallocation of tours to trucks or drones]}
	\begin{align}
		\tsp_{m',0}(W)\le \left(\left\lceil \frac{m}{m'}\right \rceil+\alpha\left\lceil \frac{n}{m'}\right \rceil\right)\tsp_{m,n}(W)           &  & \forall m'\geq 1 \label{eq:redistribution_to_trucks} \\
		\tsp_{0,n'}(W)\le \left(\left\lceil \frac{n}{n'}\right \rceil+\frac{1}{\alpha}\left\lceil \frac{m}{n'}\right \rceil\right)\tsp_{m,n}(W) &  & \forall n'\geq 1 \label{eq:redistribution_to_drones}
	\end{align}
\end{restatable}
\begin{proof}
	See Online Appendix~\ref{proof:lemma-basic}.
\end{proof}

The next lemma provides a lower bound on the optimal objective value of RDP$^*$.

\begin{restatable}{lemma}{lemboundopt}\label{lem-boundOPT}
	For any RDP$^*$ instance $I^*$ we have
	\begin{align*}
		 & \OPThome(I) \ge \max \{\tsp_{\nrtrucks,\nrdrones}(C), \tsp_{\nrtrucks,0}(D)\}
	\end{align*}
\end{restatable}
\begin{proof}
	See Online Appendix~\ref{proof:lem-boundOPT}.
\end{proof}

Based on the definition of the \tspmnname\ we can now state the \truckonly\ policy with makespan $\tsp_{\nrtrucks,0}(C)$ as: Given an instance $I=(\nrtrucks,\nrdrones,\alpha,G,D)$ with $C=V\setminus \{\depot\}$ of RDP solve \tspmnname\ using $\nrtrucks$ trucks and no drones and serve the obtained tours by the truck.

To conclude this section, we now define several specific graphs used in our examples along with several basic relations on these graphs summarized in Lemmata~\ref{lemma:relations_graphs} and \ref{lemma:relations_graphs2}.

\begin{definition}[$\tilde{\mathcal{G}}(n, d)$]
	We define $\tilde{\mathcal{G}}(n, d)$ with $n \in \mathbb{N}$ and $d \in \mathbb{R}_{> 0}$ as a \textit{star graph}, i.e.,  an undirected graph with node set $W \cup \set{\depot}$, $|W|=n$, and edges $E = \set{\set{\depot, v} \colon v \in W}$ with $c(\depot,v_i)=d$ for all $v\in D$.
\end{definition}

\begin{definition}[$\tilde{\mathcal{G}}(n_1, n_2, d_1, d_2)$]
	We define $\tilde{\mathcal{G}}(n_1, n_2, d_1, d_2)$ with $n_1, n_2 \in \mathbb{N}$ and $d_1, d_2 \in \mathbb{R}_{>0}$, as a star graph with node set $W=W_1 \cup W_2 \cup \set{\depot}$, where $\abs{W_1} = n_1$ and $\abs{W_2} = n_2$, and edge set $E = \set{\set{\depot, v} \colon v \in W_1 \cup W_2}$, with  $c(\depot,v)=d_1$ $\forall v\in W_1$, $c(\depot,v)=d_2$ $\forall v\in W_2$.
\end{definition}

\begin{figure}
	\centering
	\begin{subfigure}[b]{0.35\textwidth}
		\centering
		\includegraphics[scale=0.6]{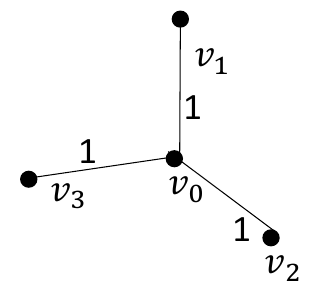}
		\caption{
			\footnotesize{$\tilde{\mathcal{G}}(3,1)$}
		}
		\label{fig:star1}
	\end{subfigure}
	\begin{subfigure}[b]{0.6\textwidth}
		\centering
		\includegraphics[scale=0.6]{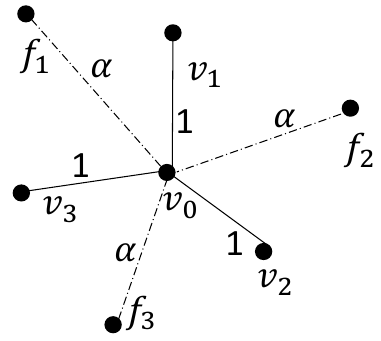}
		\caption{
			\footnotesize{$\tilde{\mathcal{G}}(3, 3,1, \alpha)$}
		}
		\label{fig:star2}
	\end{subfigure}
	\caption{\footnotesize{Examples of star graphs}}\label{fig:stargraph}
\end{figure}

Figure~\ref{fig:stargraph} illustrates graphs $\tilde{\mathcal{G}}(3, 1)$ and $\tilde{\mathcal{G}}(3,3,1,\alpha)$.

\begin{restatable}{lemma}{lemmarelationsgraphs}\label{lemma:relations_graphs}
	For graphs $\tilde{\mathcal{G}}(n,d)$ with the set of non-depot nodes $W$ and $\tilde{\mathcal{G}}(n_1, n_2,d_1, d_2)$ with the sets of non-depot nodes $W_1$ (for level-1 nodes) and $W_2$ (for level-2 nodes), $\nrtrucks\geq 1$ trucks, $\nrdrones\geq 1$ drones, and drone speed $\alpha>0$, the following relations hold

	\noindent\textit{[Truck-only and drone-only makespan]}
	\begin{align}
		\tsp_{1,0}(W')=2d|W'|                                                             &  & \forall W'\subseteq W\label{eq:one_truck_star}                         \\
		\tsp_{0,1}(W')=\frac{2}{\alpha}d|W'|                                              &  & \forall W'\subseteq W\label{eq:one_drone_star}                         \\
		\tsp_{\nrtrucks,0}(W)=2d\left\lceil\frac{n}{\nrtrucks}\right\rceil                &  & \label{eq:truck_only_star}                                             \\
		\tsp_{0,\nrdrones}(W)=\frac{2}{\alpha}d\left\lceil\frac{n}{\nrdrones}\right\rceil &  & \label{eq:drone_only_star}                                             \\
		\tsp_{1,0}(W'_1\cup W'_2)=2d_1|W'_1|+2d_2|W'_2|                                   &  & \forall W'_1\subseteq W_1, W'_2\subseteq W_2\label{eq:one_truck_star2} \\
		\tsp_{0,1}(W'_1\cup W'_2)=\frac{2}{\alpha}d_1|W'_1|+\frac{2}{\alpha}d_2|W'_2|     &  & \forall W'_1\subseteq W_1, W'_2\subseteq W_2\label{eq:one_drone_star2}
	\end{align}
\end{restatable}
\begin{proof}
	See Online Appendix~\ref{proof:lemma:relations_graphs}.
\end{proof}

\begin{lemma}\label{lemma:relations_graphs2}
	For graph $\tilde{\mathcal{G}}(n_1, n_2,1, \alpha)$ with the sets of non-depot nodes $W_1$ (for level-1 nodes) and $W_2$ (for level-2 nodes), if there is $f\in \mathbb{N}$ such that $n_1=f\nrtrucks$ and $n_2=f\nrdrones$ we have
	\begin{equation} \label{eq:mtsp_two_levelstar}
		\tsp_{\nrtrucks,\nrdrones}(W_1\cup W_2)= 2f
	\end{equation}
\end{lemma}
\begin{proof}
	We prove relation~\eqref{eq:mtsp_two_levelstar} in two steps.
	First, observe that a feasible solution, where each truck travels over $f$ nodes in $W_1$ and each drone travels over $f$ nodes in $W_2$ has the makespan of $2f$.
	Consider for example with $\nrtrucks=\nrdrones=3, f=1$ and nodes $W_1$ and $W_2$ as depicted in Figure~\ref{fig:star2}:
	In a feasible solution with the makespan of $2f=2$, drones perform tours $(\depot,f_i,\depot), i\in\{1,2,3\}$ and trucks perform tours $(\depot,v_j,\depot), j\in\{1,2,3\}$.

	Second, we prove that the lower bound on $\tsp_{\nrtrucks,\nrdrones}(W_1\cup W_2)$ is $2f$. For this, denote nodes $W^t_1\subseteq W_1$ and $W^t_2\subseteq W_2$ as the nodes visited by truck tours in some feasible solution of the respective \tspmnname. Observe that, by the triangle inequality, it suffices to consider tours that visit each node in $W'_1\cup W'_2, n'=|W'_1|+|W'_2|,$ exacly once while solving \tspmnname. Recall that the maximum over some numbers cannot be lower than the average of these numbers. Therefore, we bound $\tsp_{\nrtrucks,\nrdrones}(W_1\cup W_2)$ by the \textit{average} length of the vehicle tours; using \eqref{eq:one_truck_star2} and \eqref{eq:one_drone_star2}, we receive:
	\begin{align*}
		\tsp_{\nrtrucks,\nrdrones}(W_1\cup W_2)\geq \max\left\{\frac{2|W^t_1|+2\alpha|W^t_2|}{\nrtrucks}, \frac{\frac{2}{\alpha}\left(f\nrtrucks-|W^t_1|\right)+\frac{2}{\alpha}\alpha\left(f\nrdrones-|W^t_2|\right)}{\nrdrones}\right\}.
	\end{align*}
	Prove by contradiction. Assume that $\tsp_{\nrtrucks,\nrdrones}(W_1\cup W_2)<2f$. For this, the average lengths of the truck and drone tours of the respective optimal solution should be less than $2f$:

	\begin{align}
		\frac{2|W^t_1|+2\alpha|W^t_2|}{\nrtrucks}<2f \label{eq:hilf1} \\
		\frac{\frac{2}{\alpha}\left(f\nrtrucks-|W^t_1|\right)+\frac{2}{\alpha}\alpha\left(f\nrdrones-|W^t_2|\right)}{\nrdrones}<2f\label{eq:hilf2}
	\end{align}
	Multiply \eqref{eq:hilf1} by $\nrtrucks$ and \eqref{eq:hilf2} by $\alpha \nrdrones$ and add both inequalities, we receive:
	\begin{align*}
		2f \nrtrucks + 2 \alpha f \nrdrones
		> 2|W^t_1|+2\alpha|W^t_2|+2\left(f\nrtrucks-|W^t_1|\right)+2\alpha\left(f\nrdrones-|W^t_2|\right)
		=  2f \nrtrucks + 2 \alpha f \nrdrones
	\end{align*}
	We conclude that the bound $\tsp_{\nrtrucks,\nrdrones}(W_1\cup W_2)\geq 2f$ is proven by contradiction. Together with the feasible solution of makespan $2f$ provided above, $\tsp_{\nrtrucks,\nrdrones}(W_1\cup W_2)= 2f$.
\end{proof}

\subsection{RDP policies: \multitsphome\ and ~\singletsphome}\label{sec:policies}
This section introduces the two policies central to our analysis: ~\multitsphome\ and ~\singletsphome. Lemmata~\ref{lem-boundMultiTSP}, \ref{lem-singletsphome} and \ref{lem-singlemultiple} establish upper bounds on the policies' results.
We analyze both policies as deterministic policies.

In \multitsphome, vehicles first follow the routes of $S^*_{\nrtrucks,\nrdrones}(C)$ and return to the depot. Let $D^d\subseteq D$ be the set of damaged nodes not visited by the truck tours $S^t$ in this stage. In the \secondary\ stage, trucks traverse tours $S^*_{\nrtrucks,0}(D^d)$.
That is, the makespan of \multitsphome\ is
\begin{align}
	\multitsphome(I)=\tsp_{\nrtrucks,\nrdrones}(C)+\tsp_{\nrtrucks,0}(D^d). \label{multitsphome-makespanformula}
\end{align}

\singletsphome~ consists of \textit{initial plan} and one \textit{replanning} event for each truck:
\begin{itemize}
	\item \textit{[First stage: Initial plan.]} Compute $S^*_{\nrtrucks,0}(C)=\{s^*_i:i\in\{1,\ldots,\nrtrucks\}\}$ with makespan $\tsp_{\nrtrucks,0}(C)$.
	      Assign one truck to each tour in $S^*_{\nrtrucks,0}(C)$ and distribute the drones as evenly as possible, so that each tour receives either
	      $\left\lceil \frac{\nrdrones}{\nrtrucks} \right\rceil$ or $\left\lfloor \frac{\nrdrones}{\nrtrucks} \right\rfloor$ drones.

	      Now split each tour $s^*_i=(v_{(0)}=\depot,v_{(1)},v_{(2)},\ldots,v_{(n_i)}=v_{0})$ into up to $\nrvehicles_i$ segments, which will determine the tours of the vehicles assigned to $s^*_i$ built in a makespan minimizing way; the truck is assigned the segment adjacent to the depot and the drones serve the remaining segments. Here, $n_i$ denotes the number of non-depot nodes in  tour $s^*_i$ and
	      $\nrvehicles_i$ denotes the total number of vehicles assigned to $s^*_i$.
	      Specifically, we partition the tour in up to $\nrvehicles_i$ segments $\segment_i^l:=(v_{(j_l)},\ldots, v_{(j_{l+1} -1)})$ and complete the segments to tours
	      $\segmenttour_i^1:=(\depot,v_{(1)},\ldots, v_{(j_{2}- 1)},\depot)$,
	      $\segmenttour_i^l:=(\depot,v_{(j_{l + 1} - 1)},\ldots, v_{(j_l)},\depot)$ for $l=2,\ldots, \nrvehicles_i-1$, and
	      $\segmenttour_i^{\nrvehicles_i}:=(\depot, v_{(n_i)}, \ldots, v_{(j_{k_i})},\depot)$,
	      such that the makespan of serving tour $\segmenttour_i^1$ by truck and the other tours by drones, i.e., $\max \{ c(\segmenttour_i^1), \max_{l=2,\ldots, \nrvehicles_i} \frac{ c(\segmenttour_i^l)}{\alpha} \}$, is minimum.

	\item \textit{[Second stage: Replanning]}. For each $i$, replanning occurs when the truck reaches the last node $v_{(j_1)}$ of its segment.
	      Let $\Ddrone_i$ denote the set of damaged nodes in  $s^*_i$ not yet visited by the truck, and $C^d_i$ the set of nodes in $s^*_i$ not yet visited by any vehicle.
	      Starting from $v_{(j_1)}$ truck $i$ then follows a minimum-length tour that visits all nodes in $\Ddrone_i\cup C^d_i$ and returns to the depot.
	      The drones complete their tours.
\end{itemize}

\begin{lemma}\label{lem-boundMultiTSP}
	For any RDP instance $I$, it holds
	\begin{align}
		\multitsphome(I) \le \tsp_{\nrtrucks,\nrdrones}(C) + \tsp_{\nrtrucks,0}(D). \label{eq:ub_multitsp}
	\end{align}
\end{lemma}
\begin{proof}
	Immediate, for \eqref{multitsphome-makespanformula} and \eqref{eq:morenodes}.
\end{proof}

\begin{lemma}\label{lem-singletsphome}
	For any RDP instance $I$, it holds
	\begin{align}
		\singletsphome(I)\le \tsp_{\nrtrucks,0}(C). \label{eq:ub_singletsp}
	\end{align}
\end{lemma}
\begin{proof}
	Immediate from the definition of \singletsphome.
\end{proof}

\begin{lemma}\label{lem-singlemultiple}
	For any RDP instance $I$ and any online deterministic policy, it holds that
	\begin{align}
		\ALG(I)\geq \max\{\tsp_{\nrtrucks,\nrdrones}(C),\tsp_{\nrtrucks,0}(D)\}. \label{eq:ub_singlemultitsp}
	\end{align}
\end{lemma}
\begin{proof}
	Follows from Lemma~\ref{lem-boundOPT}.
\end{proof}

\subsection{Best-case drone-impact ratio analysis}
\label{sec-comparison_truck_bestcase}
Sections~\ref{sec:lowerbounds_omega_bestcase} and \ref{sec:upperbounds_omega_bestcase} establish the following lower and upper bounds on the best-case  drone-impact ratios:
\begin{align}
	\frac{1}{1+\alpha\left\lceil\frac{\nrdrones}{\nrtrucks}\right\rceil}\leq\underline{\omega}(\multitsphome)\leq \frac{1}{1+\alpha\frac{\nrdrones}{\nrtrucks}}  &                                                                \\
	\frac{1}{1+\alpha\left\lceil\frac{\nrdrones}{\nrtrucks}\right\rceil}\leq\underline{\omega}(\singletsphome)\leq \frac{1}{1+\alpha\frac{\nrdrones}{\nrtrucks}} & \text{\quad for }\frac{\nrdrones}{\nrtrucks}\in \mathbb{Z}^+_0
\end{align}

If the number of drones is a multiple of the number of trucks, i.e., $\frac{k^d}{k^t}\in \mathbb{Z}^+_0$, the best-case drone-impact ratios are tight:
\begin{align}
	\underline{\omega}(\multitsphome)= \underline{\omega}(\singletsphome)= \frac{1}{1+\alpha\left\lceil\frac{\nrdrones}{\nrtrucks}\right\rceil} \label{eq:tightbestratio_multitsp}
\end{align}

\subsubsection{Lower bounds on $\underline{\omega}(\ALG)$}
\label{sec:lowerbounds_omega_bestcase}
Lemma~\ref{lemma:lowerbound_omega_basecase_alg} establishes a lower bound for all deterministic online policies \ALG.

\begin{lemma}\label{lemma:lowerbound_omega_basecase_alg}
	For any $\alpha >0$, $\nrtrucks\in\mathbb{N}$, $\nrdrones\in\mathbb{N}$, and any deterministic online policy \ALG\ we have $\underline{\omega}(\ALG)\geq \frac{1}{1+\alpha\left\lceil\frac{\nrdrones}{\nrtrucks}\right\rceil}$.
\end{lemma}

\begin{proof}Using Lemma~\ref{lemma:basic} we obtain
	\begin{align}
		\underline{\omega}(\ALG)
		\geq \inf_I\frac{\tsp_{\nrtrucks,\nrdrones}(C(I))}{\tsp_{\nrtrucks,0}(C(I))}
		\geq \inf_I \frac{\tsp_{\nrtrucks,\nrdrones}(C(I))}{(1+\alpha\left\lceil\frac{\nrdrones}{\nrtrucks}\right\rceil)\tsp_{\nrtrucks,\nrdrones}(C(I))}
		\geq  \frac{1}{1+\alpha\left\lceil\frac{\nrdrones}{\nrtrucks}\right\rceil}. \label{eq:bestroneimpact}
	\end{align}
\end{proof}

\subsubsection{Upper bounds on $\underline{\omega}(\multitsphome)$ and $\underline{\omega}(\singletsphome)$}
\label{sec:upperbounds_omega_bestcase}

\begin{lemma}\label{lemma:upperbound_omega_bestcase_multitsp}
	For any $\alpha >0$, $\nrtrucks\in\mathbb{N}$, $\nrdrones\in\mathbb{N}$, we have
	$\underline{\omega}(\multitsphome)\leq \frac{1}{1+\alpha\frac{\nrdrones}{\nrtrucks}}$.
\end{lemma}
\begin{proof}
	Consider an instance $I$ with \textit{no damaged nodes} ($D=\emptyset$), in which the depot node $\depot$ and nodes $C$ form a graph $\tilde{\mathcal{G}}(\left(\nrtrucks\right)^2, \nrdrones\nrtrucks,1, \alpha)$, i.e., a graph with  $\left(\nrtrucks\right)^2$ level-1 nodes and $\nrdrones\nrtrucks$ level-2 nodes. From \eqref{eq:mtsp_two_levelstar},
	$\tsp_{\nrtrucks,\nrdrones}(C)= 2\nrtrucks$.

	We receive makespan of an optimal truck-only solution by evenly distributing the nodes among trucks and using \eqref{eq:one_truck_star2} (to recheck this argument, recall that the maximum over some numbers cannot be larger than the average over these numbers and compute the average tour length over the trucks): $\tsp_{\nrtrucks,0}(C)= 2\nrtrucks+2\alpha \nrdrones$.

	From these expressions for $\tsp_{\nrtrucks,\nrdrones}(C)$ and $\tsp_{\nrtrucks,0}(C)$, it follows:
	\begin{align}
		\underline{\omega}(\multitsphome)\leq\frac{\multitsphome(I)}{\tsp_{\nrtrucks,0}(C)}= \frac{\tsp_{\nrtrucks,\nrdrones}(C)}{\tsp_{\nrtrucks,0}(C)}=\frac{1}{1+\alpha\frac{\nrdrones}{\nrtrucks}}
	\end{align}
\end{proof}

\begin{restatable}{lemma}{lemmaupperboundimegabestcasesingletsp}\label{lemma:upperbound_omega_bestcase_singletsp}
	If $\frac{\nrdrones}{\nrtrucks}\in \mathbb{Z}^+_0$, for any $\alpha >0$, we have $\underline{\omega}(\singletsphome)\leq \frac{1}{1+\alpha\frac{\nrdrones}{\nrtrucks}}$.
\end{restatable}
\begin{proof}
	See Online Appendix~\ref{proof:lemma:upperbound_omega_bestcase_singletsp}.
\end{proof}

\subsection{Worst-case drone-impact ratio analysis}
\label{sec-comparison_truck_worstcase}
Observe that for any online policy $\ALG$, $\bar{\omega}(\ALG)\geq 1$, since it takes at least $\tsp_{\nrtrucks,0}(C)$ to visit the damaged nodes for instances with all nodes damaged ($D=C$).

Sections~\ref{sec:upperbounds_omega_worstcase} and \ref{sec:lowerbounds_omega_worstcase} establish lower and upper bounds on the worst-case drone-impact ratios. Since $\bar{\omega}(\singletsphome)=1$ (Lemma~\ref{thm:upperbound_omega_worstcase_singletsp}), no online policy can achieve a better worst-case drone-impact ratio, and under \singletsphome\, the use of drones never degrades performance of emergency deliveries relative to truck-only operations.

For \multitsphome\ we show that
$\min\{1+\frac{1}{\alpha}, \frac{1+\alpha}{1+\epsilon}\}\le \bar{\omega}(\multitsphome)\leq \min\{2,\frac{1}{\alpha} \left\lceil\frac{\nrtrucks}{\nrdrones}\right\rceil+1\}$ for an arbitrarily small $\varepsilon>0$. The bound is \textit{tight} for fast drones ($\alpha>1$), if there are not too many trucks, e.g., $k^t\leq k^d$. For $\alpha=1$, the bound is nearly tight for all $k^t$ and $k^d$, in the sense that the worst-case instance $I$ in Lemma~\ref{lemma:lowerbound_omega_worstcase_multitsp_new} yields ratios $\frac{\multitsphome(I)}{\tsp_{\nrtrucks,0}(C)}$ arbitrarily close to the theoretical bound of 2.

\subsubsection{Upper bounds on $\bar{\omega}(\multitsphome)$ and $\bar{\omega}(\singletsphome)$}
\label{sec:upperbounds_omega_worstcase}

\begin{lemma}\label{thm:upperbound_omega_worstcase_multitsp}
	For any $\alpha >0$, $\nrtrucks\in\mathbb{N}$, $\nrdrones\in\mathbb{N}$, we have $\bar{\omega}(\multitsphome)\leq \min\{2,\frac{1}{\alpha} \left\lceil\frac{\nrtrucks}{\nrdrones}\right\rceil+1\}$.
\end{lemma}
\begin{proof}
	On the one hand, from the definition of \multitsphome, \eqref{eq:morenodes}, and \eqref{eq:moretrucks}, it follows
	\begin{align*}
		\bar{\omega}(\multitsphome) =\sup_I \frac{\multitsphome(I)}{\tsp_{\nrtrucks,0}(C)}
		= \frac{\tsp_{\nrtrucks,\nrdrones}(C) + \tsp_{\nrtrucks,0}(\Ddrone)}
		{ \tsp_{\nrtrucks,0}(C)}\leq \\
		\le  \frac{\tsp_{\nrtrucks,0}(C) +\tsp_{\nrtrucks,0}(C)}
		{ \tsp_{\nrtrucks,0}(C)}=2.
	\end{align*}

	On the other hand, from the definition of \multitsphome, \eqref{eq:morenodes}, \eqref{eq:moretrucks}, and \eqref{eq:redistribution_to_drones}, it follows
	\begin{align*}
		\bar{\omega}(\multitsphome) =\sup_I \frac{\multitsphome(I)}{\tsp_{\nrtrucks,0}(C)}
		= \frac{\tsp_{\nrtrucks,\nrdrones}(C) + \tsp_{\nrtrucks,0}(\Ddrone)}
		{ \tsp_{\nrtrucks,0}(C)}\leq \\
		\le  \frac{\tsp_{0,\nrdrones}(C) +\tsp_{\nrtrucks,0}(C)}
		{ \tsp_{\nrtrucks,0}(C)}\leq  \frac{\frac{1}{\alpha} \left\lceil\frac{\nrtrucks}{\nrdrones}\right\rceil\tsp_{\nrtrucks,0}(C) +\tsp_{\nrtrucks,0}(C)}
		{ \tsp_{\nrtrucks,0}(C)} =   \\
		=\frac{1}{\alpha} \left\lceil\frac{\nrtrucks}{\nrdrones}\right\rceil+1.
	\end{align*}
\end{proof}

\begin{lemma}\label{thm:upperbound_omega_worstcase_singletsp}
	For any $\alpha >0$, $\nrtrucks\in\mathbb{N}$, $\nrdrones\in\mathbb{N}$, we have $\bar{\omega}(\singletsphome)\leq 1$.
\end{lemma}
\begin{proof}
	Immediately follows from Lemma~\ref{lem-singletsphome}.
\end{proof}

\subsubsection{Lower bounds on $\bar{\omega}(\multitsphome)$ and $\bar{\omega}(\singletsphome)$}
\label{sec:lowerbounds_omega_worstcase}

\begin{restatable}{lemma}{lemmalowerboundomegaworstcasemultitspnew}\label{lemma:lowerbound_omega_worstcase_multitsp_new}
	For any $\alpha>0$, $\nrtrucks\in\mathbb{N}$, $\nrdrones\in\mathbb{N}$, there is an instance $I$ with
	$\frac{\multitsphome(I)}{\tsp_{\nrtrucks,0}(C)}=\min\{1+\frac{1}{\alpha}, \frac{1+\alpha}{1+\epsilon}\}$.
\end{restatable}
\begin{proof}
	See Online Appendix~\ref{proof:lemma:lowerbound_omega_worstcase_multitsp_new}.
\end{proof}

\begin{lemma}\label{lemma:lowerbound_omega_worstcase_singletsp}
	For any $\alpha >0$, $\nrtrucks\in\mathbb{N}$, $\nrdrones\in\mathbb{N}$, we have $\bar{\omega}(\singletsphome)= 1$.
\end{lemma}
\begin{proof}
	For instances $I$ with $D=C$ (all nodes are damaged), $\frac{\singletsphome(I)}{\tsp_{\nrtrucks,0}(C)}=1$. Lemma \eqref{lem-singletsphome} completes the proof.
\end{proof}

\subsection{Competitive ratio analysis}\label{sec:competitive}
Section~\ref{sec:upperboundsigma} provides upper bounds on the competitive ratios, while Sections~\ref{sec:worstcase}, \ref{sec-lowerbounds_singletsp}, and \ref{sec-lowerbounds} establish lower bounds for $\multitsphome$, $\singletsphome$, and any deterministic online policy \ALG, respectively.

Lemmata~\ref{cor-upperbounds} and \ref{lem:alpha=1var2}-\ref{lemma:wc-ratiosmaller2new},  establish the following \textit{tight} competitive ratios for $\multitsphome$:
\begin{align}
	\sigma(\multitsphome)= \min\{2,1+\left\lceil \frac{\nrdrones}{\nrtrucks}\right\rceil \alpha\} \label{eq:tightsigma_multitsp}
\end{align}

No online deterministic policy can achieve a better competitive ratio than $\multitsphome$, if trucks and drones have the same speed or if the truck speed is an integer multiple of the drone speed (Lemma~\ref{lem-1/alphainN}).
Moreover, if the  drone speed is an integer multiple of the truck speed, the competitive-ratio lower bound for any deterministic policy is $2-\frac{2}{\alpha}$; for increasing drone speed,  this bound approaches the competitive ratio of $\multitsphome$.

For \singletsphome\ we are able to establish tightness of the competitive ratio $\sigma(\singletsphome)$ $= 1+\left\lceil \frac{\nrdrones}{\nrtrucks}\right\rceil \alpha$,  if $\alpha=1$ and $\nrtrucks\ge \alpha\left\lceil\frac{\nrdrones}{\nrtrucks}\right\rceil$ (Lemma~\ref{lemma:wc-alpha1}) or if $\frac{1}{\alpha}$ is integer and $\nrtrucks\ge \left\lceil\frac{\nrdrones}{\nrtrucks}\right\rceil+b-1$ (Lemma~\ref{lemma:wc-bigalpha1-1}).
For $\alpha\in \mathbb{N}$ and $\nrtrucks\ge \left\lceil\frac{\alpha\nrdrones}{\nrtrucks}\right\rceil$, we prove a slightly weaker result  $\sigma(\singletsphome)\ge 1+\left\lceil \frac{\alpha\nrdrones}{\nrtrucks}\right\rceil$ (Lemma~\ref{lemma:wc-bigalpha1part2-singletsp}).

\subsubsection{Upper bounds on $\sigma(\multitsphome)$ and $\sigma(\singletsphome)$}\label{sec:upperboundsigma}

\begin{restatable}{lemma}{lemmacorupperbounds}\label{cor-upperbounds}
	For any $\alpha >0$, $\nrtrucks\in\mathbb{N}$, $\nrdrones\in\mathbb{N}$, we have $\sigma(\multitsphome)\le \min\{2,1+\left\lceil \frac{\nrdrones}{\nrtrucks}\right\rceil \alpha\}$.
\end{restatable}
\begin{proof}
	See Online Appendix~\ref{proof:cor-upperbounds}.
\end{proof}

\begin{restatable}{lemma}{lemmacorupperboundssingletsp}\label{cor-upperbounds_singletsp}
	For any $\alpha >0$, $\nrtrucks\in\mathbb{N}$, $\nrdrones\in\mathbb{N}$, we have $\sigma(\singletsphome)\le 1+\left\lceil \frac{\nrdrones}{\nrtrucks}\right\rceil \alpha$.
\end{restatable}
\begin{proof}
	See Online Appendix~\ref{proof:cor-upperbounds_singletsp}.
\end{proof}

\subsubsection{Lower bounds on $\sigma(\multitsphome$)}\label{sec:worstcase}

In this section, Lemmata~\ref{lem:alpha=1var2}-\ref{lemma:wc-ratiosmaller2new} construct  worst-case instances for \textit{all possible} values of $\alpha, k^t$, and $k^d$, and illustrate that the bounds from Lemma~\ref{cor-upperbounds} are tight. The ratios of Lemma~\ref{cor-upperbounds} are $\min\{2,1+\left\lceil \frac{\nrdrones}{\nrtrucks}\right\rceil \alpha\}$. The cases consider $\alpha \left\lceil\frac{\nrdrones}{\nrtrucks} \right\rceil< 1$ (Lemma~\ref{lemma:wc-ratiosmaller2new}); and $\alpha<1$  and $\alpha \left\lceil\frac{\nrdrones}{\nrtrucks} \right\rceil\geq 1$ (Lemma~\ref{lemma:wc-bigalpha1-2}). The cases of $\alpha \left\lceil\frac{\nrdrones}{\nrtrucks} \right\rceil\geq 1$ and $\alpha=1$ as well $\alpha \left\lceil\frac{\nrdrones}{\nrtrucks} \right\rceil\geq 1$ and $\alpha>1$ are covered by Lemmata~\ref{lem:alpha=1var2}  and \ref{lemma:kdkv_LB_2new2}, respectively.

\begin{lemma}\label{lem:alpha=1var2}
	For $\alpha= 1$, $\nrdrones, \nrtrucks\in \mathbb{N}$, there is an instance $I$ with $\frac{\multitsphome(I)}{\OPT(I^*)}=2$.
\end{lemma}
\begin{proof}
	Consider an instance with one damaged node, in which the depot node $\depot$ and nodes $C$ form a graph $\tilde{\mathcal{G}}(\nrtrucks+\nrdrones,1)$. In $\multitsphome$, each vehicle first performs round tours from the depot $(v_0, v, v_0), v\in C$ with the makespan of 2.
	Observe that notwithstanding which nodes  $\multitsphome$ decides to be visited by the trucks, there is an instance in which the damaged node is visited by the drone.
	For this instance, $\multitsphome(I)=4$ as a truck has to revisit the damaged node in the \secondary\ stage.

	Observe that $OPT(I^*)=2$, because a truck visit the damaged node, and $\frac{\multitsphome(I)}{\OPT(I^*)}=2$.

	Figure~\ref{fig:star1} illustrates such an instance with $k^t=1$ and $k^d=2$. Regardless of which nodes $v \in \{v_1, v_2, v_3\}$ are visited initially by the drones in $\multitsphome$, the oracle can always select an instance with the damage placed at one of these nodes.
\end{proof}

\begin{restatable}{lemma}{lemmakdkvlbnew}\label{lemma:kdkv_LB_2new2}
	For $\alpha\neq 1$, $\nrdrones, \nrtrucks \in \mathbb{N}$, there is an instance $I$ with $\frac{\multitsphome(I)}{\OPT(I^*)}= \min \{2,1+\alpha\}$.
\end{restatable}
\begin{proof}
	See Online Appendix~\ref{proof:lemma:kdkv_LB_2new2}.
\end{proof}

\begin{restatable}{lemma}{lemmawcbigalphaoptimistic}\label{lemma:wc-bigalpha1-2}
	For any $\alpha<1,\nrtrucks,\nrdrones$ such that $\alpha \left\lceil\frac{\nrdrones}{\nrtrucks} \right\rceil\ge 1$, there exists and instance $I$ with
	$\frac{\multitsphome(I)}{\OPT(I^*)}=2$.
\end{restatable}
\begin{proof}
	See Online Appendix~\ref{proof:lemma:wc-bigalpha1optimistic}.
\end{proof}

\begin{restatable}{lemma}{lammawcratiosmaller}\label{lemma:wc-ratiosmaller2new}
	For any $\alpha>0 $ and $\nrtrucks, \nrdrones \in \mathbb{N} $ with $\alpha\left\lceil \frac{\nrdrones}{\nrtrucks}\right \rceil<1$,
	there exists an instance $I$ with $\frac{\multitsphome(I)}{\OPT(I^*)} = 1+\left \lceil \frac{\nrdrones}{\nrtrucks}\right \rceil\alpha$.
\end{restatable}
\begin{proof}
	See Online Appendix~\ref{proof:lemma:wc-ratiosmaller2new}.
\end{proof}

\subsubsection{Lower bounds on $\sigma(\singletsphome)$}\label{sec-lowerbounds_singletsp}

\begin{lemma}\label{lemma:wc-alpha1}
	For $\alpha=1$, $\nrdrones\in \mathbb{N}$ and $\nrtrucks\ge \alpha\left\lceil\frac{\nrdrones}{\nrtrucks}\right\rceil$
	there exists an instance $I$ with $\frac{\singletsphome(I)}{\OPT(I^*)}=1+\alpha \left\lceil \frac{\nrdrones}{\nrtrucks}\right\rceil$.
\end{lemma}
\begin{proof}
	Consider an instance $I$ with depot node $v_0$ and nodes $C$ forming a star graph $\gstar(\nrtrucks+ \nrdrones,1)$.
	Observe that if up to $\nrtrucks$ nodes are damaged, using \eqref{eq:mtsp_two_levelstar} we have $\OPT(I^*)=2$.

	In the \initial\ phase of $\singletsphome$, truck tours that visit all nodes are built. Each of these tours contains either $1+\alpha\left\lfloor\frac{\nrdrones}{\nrtrucks}\right\rfloor$ or $1+\alpha\left\lceil\frac{\nrdrones}{\nrtrucks}\right\rceil$ non-depot nodes, and there is at least one tour with exactly $1+\alpha\left\lceil\frac{\nrdrones}{\nrtrucks}\right\rceil$ non-depot nodes and makespan $\tsp_{\nrtrucks,0}(C)=2(1+\alpha\left\lceil\frac{\nrdrones}{\nrtrucks}\right\rceil)$.
	Consider an instance $I$, in which all but the first non-depot node
	initially assigned to one such truck (and its supporting drones) are damaged. Then, $\singletsphome= 2(1+\left\lceil\frac{\nrdrones}{\nrtrucks}\right\rceil)$, and $\frac{\singletsphome(I)}{\OPT(I^*)}=1+\alpha \left\lceil \frac{\nrdrones}{\nrtrucks}\right\rceil$.
\end{proof}

\begin{restatable}{lemma}{lemmawcbigalpha}\label{lemma:wc-bigalpha1-1}
	For any $\alpha=\frac{1}{b}, b\in \mathbb{N},\nrtrucks,\nrdrones\in\mathbb{N}$ with $\nrtrucks\ge \left\lceil\frac{\nrdrones}{\nrtrucks}\right\rceil+b-1$,
	there exists an instance $I$ with $\frac{\singletsphome(I)}{\OPT(I^*)}=1+\alpha \left\lceil \frac{\nrdrones}{\nrtrucks}\right\rceil$.
\end{restatable}
\begin{proof}
	See Online Appendix~\ref{proof:lemma:wc-bigalpha1-1}.
\end{proof}

\begin{restatable}{lemma}{lemmawcbigalphasingletsp}\label{lemma:wc-bigalpha1part2-singletsp}
	For any $\alpha\in \mathbb{N},\nrtrucks,\nrdrones\in\mathbb{N}$ with $\nrtrucks\ge \left\lceil\frac{\alpha\nrdrones}{\nrtrucks}\right\rceil$,
	there exists an instance $I$ with $\frac{\singletsphome(I)}{\OPT(I^*)}=1+\left\lceil \frac{\alpha\nrdrones}{\nrtrucks}\right\rceil$.
\end{restatable}
\begin{proof}
	See Online Appendix~\ref{proof:lemma:wc-bigalpha1part2-singletsp}.
\end{proof}

\subsubsection{Lower bounds on the competitive ratio for any online policy $\ALG$}\label{sec-lowerbounds}

\begin{lemma}\label{lem-1/alphainN}
	Consider $\nrtrucks, \nrdrones \in \NN$ and $\alpha>0$ with $\frac{1}{\alpha} \in \mathbb{N}$, i.e., the truck speed is an integer multiple of the drone speed.
	For any online policy $\policy$, there is an instance $I$ such that $\frac{ALG(I)}{\OPT(I^*)} \ge \min \{2,1+ \left \lceil \frac{\nrdrones}{\nrtrucks} \right\rceil \alpha\}$.
\end{lemma}

\begin{proof}
	Let $b:=\frac{1}{\alpha}$. Consider an instance $I$  with $|D|\le \nrtrucks b$ and the depot node $v_0$ and non-depot nodes $C$ forming a star graph $\gstar(\nrdrones+b\nrtrucks,1)$.
	Let  us first look at the full information situation.
	If $|D|\le \nrtrucks b$, each drone visits one non-depot node and each truck visits $b$ non-depot nodes in an optimal solution, resulting in a makespan of $\OPThome(I^*)=\max\{2b,2b\}=2b$.

	Let $\policy$ be an  online policy with a competitive $\sigma(ALG)<2$.
	In the constructed instance, denote the total number of non-depot nodes visited by drones in $\policy$ as $n'$.
	Observe that $n'\leq k^d$.
	Note that if a drone visits more than one non-depot node, its tour duration would be $\ge 2 \frac{2}{\alpha}= 2\cdot 2b$, leading to a competitive ratio of at least $2$.

	Assume that $\min\{n', b\nrtrucks\}$ of the nodes visited by drones are damaged.
	That means that these need to be visited by trucks as well. Consider two cases:

	If $n'\le b\nrtrucks$, this means that in \policy\ all nodes are visited by trucks: trucks visit nodes $(k^d+bk^t-n')$ and revisit $\min\{n', b\nrtrucks\}$ damaged nodes  originally visited by the drones. Since $\nrtrucks$ trucks visit $\nrdrones+b\nrtrucks$ nodes,  there is at least one truck that visits $\left\lceil \frac{\nrdrones+b\nrtrucks}{\nrtrucks}\right\rceil=  b + \left\lceil\frac{\nrdrones}{\nrtrucks}\right\rceil$ nodes. From \eqref{eq:one_truck_star},  $\policy(I)\ge 2(b + \left\lceil\frac{\nrdrones}{\nrtrucks}\right\rceil)$ and $\frac{ALG(I)}{\OPT(I^*)}\ge 1+ \left\lceil\frac{\nrdrones}{\nrtrucks}\right\rceil \alpha$.

	If $n'> b\nrtrucks$, this means that the trucks visit in total $2b\nrtrucks$ nodes in $\policy$, leading to a makespan of $4b$ and a competitive ratio of $\frac{ALG(I)}{\OPT(I^*)}\ge 2$.
\end{proof}

\begin{restatable}{lemma}{lemalphainn}\label{lem-alphainN}
	Consider $\nrtrucks, \nrdrones \in \NN$ and $\alpha\in \mathbb{N}$, i.e., the drone speed is an integer multiple of the truck speed.
	For any policy $\policy$ there is an instance $I$ such that $\frac{ALG(I)}{\OPT(I^*)}\ge 2-\frac{1}{2\alpha}$.
\end{restatable}
\begin{proof}
	See Online Appendix~\ref{proof:lem-alphainN}.
\end{proof}

Lemma~\ref{lem-1/alphainN} and Lemma~\ref{lem-alphainN} only apply to integer drone speed $\alpha$ (or $\alpha$ with $\frac{1}{\alpha}$, respectively), because the graph constructions in the respective examples are only possible in these cases. Note that we can, however, also analyze the $\multitsphome$ policy on the there-constructed graphs if $\alpha$ does not take the value that the construction was meant to, obtaining (slightly) weaker bounds.

Lemma~\ref{lem-alpha1-genlower-new} relies on a similar idea: on the instance constructed there (in dependence of $\nrtrucks$ and $\nrdrones$) we can prove a tight bound if $\alpha=1$, but it also provides bounds for other values of $\alpha$.

\begin{restatable}{lemma}{lemalphagenlowernew}\label{lem-alpha1-genlower-new}
	Consider $\nrtrucks, \nrdrones \in \NN$ and $\alpha>0$.
	For any policy $\policy$ there is an instance $I$ with $\nrtrucks$ trucks, $\nrdrones$ drones, and drone speed $\alpha$ such that $\frac{ALG(I)}{\OPT(I^*)}\ge \min\{\frac{2}{\alpha},2\alpha\}$.
\end{restatable}
\begin{proof}
	See Online Appendix~\ref{proof:lem-alpha1-genlower-new}.
\end{proof}

\section{Experimental evaluation}\label{sec:experiments}
This section reports the \textit{observed} performance of the $\multitsphome$ and $\singletsphome$ based on extensive computational experiments.
We first describe benchmark instances in Section~\ref{sub:instances}. Section~\ref{sub:comp-ratio-analysis} presents the observed competitive ratios results and Section~\ref{sub:risk-analysis} outlines the observed drone-impact ratios.
We emphasize the scale of these experiments. The evaluated online policies $\multitsphome$ and $\singletsphome$, as well as the benchmark policies, require solving the strongly NP-hard problem \tspmnname\ repeatedly: once for the initial plan and after each replanning.

\subsection{Benchmark instances} \label{sub:instances}
We generate five \textit{classes} of \emph{benchmark graphs}: \emph{Random}, \emph{1-Center}, \emph{2-Center}, \emph{Coastal}, and \emph{Mountain}. The \emph{Random} class consists of complete graphs generated by sampling
node locations uniformly from the unit square and setting the first generated node as the depot.
\emph{1-Center} and \emph{2-Center}  graphs are generated following the approach of \citep{opt-approaches-for-tsp-with-drones}. \emph{1-Center} graphs have one central node, which serves as the depot, whereas \emph{2-Center} graphs have two central nodes and the depot is placed at one of them. These three classes of graphs are Euclidean, with edge weights given by Euclidean distances. \emph{Coastal} and \emph{Mountain} classes consist of metric planar graphs designed to resemble transportation networks in coastal and mountainous regions, respectively. In \emph{Coastal} graphs, villages are concentrated near a coastline, with inland villages typically aligned perpendicular to it.
In \emph{Mountain} graphs, villages are primarily located in valleys, with fewer villages at higher elevations. Online Appendix~\ref{sub:benchmarkgraphs} provides details on the graph generation process.
For each graph, the set of damaged nodes is generated according to a \textit{damage probability} $\delta$, i.e., each node is damaged with probability $\delta$.

Based on these classes of graphs, we construct four sets of benchmark instances.
\IRANDOM\ comprises 80 \emph{Random} graphs, with 20 graphs generated for each instance size of $n \in {13, 15, 18, 21}$ nodes. For each graph, we randomly independently construct five sets of damaged nodes with damage probabilities $\delta \in \{0.1, 0.3, 0.5, 0.7, 0.9\}$. We consider drone speeds of $\alpha \in \set{0.25, 0.5, 1.0, 2.0, 4.0}$ and the use of one truck and one drone, i.e., $\nrtrucks=\nrdrones=1$.
Instance set \IBASE\ contains $20$ \emph{Random} graphs (and corresponding sets of damaged nodes) with $18$ nodes and damage probability of $0.3$.
We consider drone speeds of $\alpha \in \set{0.25, 0.5, 1.0, 2.0, 4.0}$ and the use of one truck and one drone, i.e., $\nrtrucks=\nrdrones=1$.
\IGRAPHCLASS\ consists of instances of \IBASE\ and similarly generated instances of other graph classes. For each graph class, $20$ graphs with $n=18$ nodes and damage probability $\delta=0.3$ are generated, thereby the  \emph{Coastal} and \emph{Mountain} graphs have  36 edges.
We consider drone speeds of  $\alpha \in \set{0.25, 0.5, 1.0, 2.0, 4.0}$ and the use of one truck and one drone, i.e., $\nrtrucks=\nrdrones=1$.
\ILARGE\ contains the instances from \IRANDOM\ and \IGRAPHCLASS.
Finally, for the computational experiments with more than two vehicles, where solution times are considerably higher than for the case of one truck and one drone, we use instance set \ISMALL\ that contains the $13$-node graphs (and corresponding sets of damaged nodes) from \IRANDOM.
We consider drone speeds of $\alpha \in \set{0.5, 1.0, 2.0}$ and up to four vehicles, i.e., $(\nrtrucks,\nrdrones) \in \{(1,1),(1,2),(1,3),(2,1),(2,2),(3,1)\}$.

All benchmark graphs used for testing and respective generators are provided in \cite{neugebauer2026gengraph,neugebauer2026implementation}. The computational experiments were executed on a Ryzen 1700x with 8GB of RAM on Void Linux $6.12.50$ and the implementational details of the tested online policies are explained in Online Appendix~\ref{sec:implementation}.

\subsection{Competitive Ratio Analysis} \label{sub:comp-ratio-analysis}
This section investigates the performance of \multitsphome\ and \singletsphome\, focusing on whether the actually observed competitive ratios are close to the lower and upper bounds proven in Section~\ref{sec:analysis}. Section~\ref{sec:1t1d} reports  results for $k^t=1$ truck and $k^d=1$ drone on the \ILARGE\ dataset. Section~\ref{sec:sensitivity} investigates the case of $k^t=1$ truck and $k^d=1$ drone in more detail on \IBASE, analyzing the effects of damage probability, graph class, and instance size using a one-factor-at-a-time experimental design. Comparisons with alternative online policies are presented in Section~\ref{sec:comp_bench}, and Section~\ref{sub:comp-ratio-analysis:multiple} concludes with results for multiple trucks and drones.

We denote the worst observed competitive ratio as $\hat{\sigma}$ and the the median observed competitive ratio as $\sigma_{\text{med}}$.

\subsubsection{Competitive ratio for one Truck and Drone}\label{sec:1t1d}
\begin{figure}
	\centering
	\footnotesize
	\begin{tikzpicture}
		\pgfplotstableread[col sep=comma]{ratio_vs_alpha_multitsp.csv}\ratioVSalphaMultitsp
		\pgfplotstabletranspose\datatransposed{\ratioVSalphaMultitsp}
		\pgfplotstableread[col sep=comma]{ratio_vs_alpha_singletsp.csv}\ratioVSalphaSingletsp
		\pgfplotstabletranspose\datatransposedSingle{\ratioVSalphaSingletsp}
		\begin{groupplot}[
				group style = {
						group size= 2 by 1,
						ylabels at=edge left,
					},
				width=0.5\textwidth,
				height=0.33\textwidth,
				boxplot/draw direction = y,
				boxplot/box extend=0.2,
				x axis line style = {opacity=1},
				enlarge y limits,
				ymajorgrids,
				xtick = {0.25, 1, 2, 4},
				xticklabel style = {align=center, font=\footnotesize},
				extra x ticks = {0.5},
				extra x tick style={xticklabel style={yshift=-10pt, font=\small}},
				xlabel = {$\alpha$},
				xmin = 0,
				xmax = 4.75,
				ylabel = {Competitive Ratio},
				ytick = {1, 1.5, 2, 2.5, 3, 3.5, 4, 4.5},
				ymin = 1,
				ymax = 4.5,
				every boxplot/.style={
						draw=black,
						solid,
						mark=o,
					},
			]
			\nextgroupplot[title={$\multitsphome$}]
			\addplot+[boxplot={draw position=0.25}, fill=color-a0.25] table[y index=1] {\datatransposed};
			\addplot+[boxplot={draw position=0.5}, fill=color-a0.5] table[y index=2] {\datatransposed};
			\addplot+[boxplot={draw position=1}, fill=color-a1] table[y index=3] {\datatransposed};
			\addplot+[boxplot={draw position=2}, fill=color-a2] table[y index=4] {\datatransposed};
			\addplot+[boxplot={draw position=4}, fill=color-a4] table[y index=5] {\datatransposed};
			\addplot[draw=red, mark=] coordinates {(0, 1) (1, 2) (4.5, 2)};
			\nextgroupplot[title={$\singletsphome$}]
			\addplot+[boxplot={draw position=0.25}, fill=color-a0.25] table[y index=1] {\datatransposedSingle};
			\addplot+[boxplot={draw position=0.5}, fill=color-a0.5] table[y index=2] {\datatransposedSingle};
			\addplot+[boxplot={draw position=1}, fill=color-a1] table[y index=3] {\datatransposedSingle};
			\addplot+[boxplot={draw position=2}, fill=color-a2] table[y index=4] {\datatransposedSingle};
			\addplot+[boxplot={draw position=4}, fill=color-a4] table[y index=5] {\datatransposedSingle};
			\addplot[draw=red, mark=] coordinates {(0, 1) (4, 5)};
		\end{groupplot}
	\end{tikzpicture}
	\caption{
		\footnotesize{Observed competitive ratios $\hat{\sigma}$ for $\multitsphome$ and $\singletsphome$ on \ILARGE\ \\
			\textit{Note.} The red line indicates analytical competitive ratios $\sigma$ from Section~\ref{sec:analysis}.}
	}
	\label{fig:ratio11}
\end{figure}
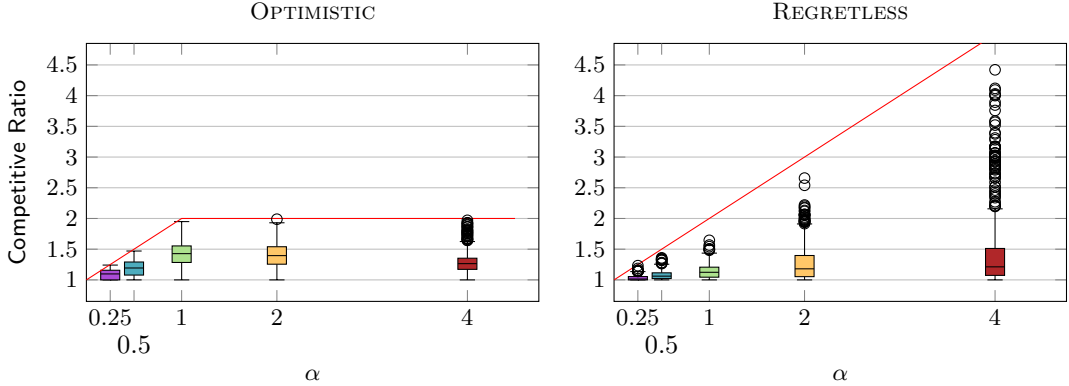
In Figure~\ref{fig:ratio11}, red lines indicate the upper and lower bounds for the competitive ratios established in Section~\ref{sec:analysis}.
The lower and upper bounds of the boxes indicate the $25\%$ and $75\%$ quartile of the input data.
The line within a box gives the median value.
The upper whiskers terminate at the largest value smaller than the $75\%$ quantile plus $1.5$ times the interquartile range.
The lower whiskers are drawn analogously.
Further outliers are drawn as a circle.

For \multitsphome\ for all drone speeds $\alpha$ we observe a wide range of competitive ratios, spanning (almost) the whole range between $1$ (optimal performance) and the red line (worst-case performance).
\singletsphome\ performs favorably for slow drones ($\alpha\le 1$), where both $\hat{\sigma}$ and $\sigma_{\text{med}}$ are consistently lower for \singletsphome\ than for \multitsphome. In contrast, for fast drones ($\alpha > 1$), \multitsphome\ becomes advantageous: $\hat{\sigma}(\singletsphome)$ increases rapidly with $\alpha$ and exceeds 2 -- which is the competitive ratio of \multitsphome\ --already starting from $\alpha = 2$.

\subsubsection{Sensitivity analysis for one truck and one drone}\label{sec:sensitivity}
Table~\ref{tab:cr-vs-dp} examines the effects of damage probabilities, graph classes, and instance sizes in detail in a one-factor-at-a-time design around the \IBASE\ dataset (highlighted in grey).

Lower \textit{damage probabilities} appear to increase $\hat{\sigma}$ and $\sigma_{\text{med}}$ in case of fast drones ( from $\alpha>1$ for \multitsphome\ and from $\alpha\ge 1$ for \singletsphome). This is especially interesting, remembering that the worst-case instance constructed in Lemma~\ref{lemma:kdkv_LB_2new2} for \multitsphome\ had one damaged node.
\singletsphome\ achieves higher ratios for small damage percentages and, as expected, \singletsphome\ outperforms \multitsphome\ for large damage probabilities. Indeed, the truck's tour in an $\text{RDP}^*$ solution closes in on a $\tsp$ tour favoring $\singletsphome$ over $\multitsphome$.

Both \singletsphome\ and \multitsphome\ appear robust with respect to different graph classes and instance sizes, as no significant performance differences are observed across these settings.

\subsubsection{Comparison to benchmark policies}\label{sec:comp_bench}
Table~\ref{tab:cr-vs-policy} contrasts the performance of \multitsphome\ and \singletsphome\ to alternative intuitive online policies: 1.) \truckonly\, where only trucks visit the nodes, 2.) the \emph{Explore-First-Help-Second policy} (\efhshome),  where the drones visit all nodes  in the \initial\ step and, as soon as all drones have returned to the depot, the trucks visit the damaged nodes, and 3.) a variant of \efhshome\, called \emph{Explore-First-Help-ASAP} (\efhahome), where the drones start exploring and the trucks follow immediately after a damaged node is discovered by a drone.

For slow drones, \singletsphome\ outperforms the alternative policies, with \truckonly\ following at a small margin.
As expected, \efhshome\ and \efhahome\ perform poorly in this setting.

This picture changes with increasing drone speed. For $\alpha\in \{2,4\}$, \truckonly\ takes significantly longer than the remaining drone-based policies, and \efhahome\ becomes the best policy both with respect to the worst observed and median competitive ratios.
Note, however, that the competitive ratio of $\efhshome$ and $\efhahome$ may exceed $2$ for $\alpha\geq 1$ -- which is the competitive ratio for $\multitsphome$ in this case (see Lemma~\ref{lemma:efhshome}).

\begin{lemma}\label{lemma:efhshome}
	For $\alpha\in \NN$ and $k^t=k^d=1$, there are instances $I$ and $I'$, such that $\frac{\efhshome(I)}{OPT(I^*)}\ge 2+\frac{1}{2\alpha}$ and $\frac{\efhahome(I)}{OPT(I^*)}\ge 2+\frac{1}{2\alpha}$.
\end{lemma}
\begin{proof}
	Consider instances $I$ and $I'$, in which the depot node $v_0$ and nodes $C$ form a  star graph $\gstar(\alpha+1, 1)$ and assume that one node is damaged. The optimal RDP$^*$ solution with makespan $2$ sends the truck to visit the damaged node and the drone to visit all other nodes. In contrast, under $\efhshome$ and $\efhahome$, the damaged node is discovered only when the drone reaches it, at time $\frac{2\alpha+1}{\alpha}$. Even if the truck is dispatched immediately (\efhahome), the resulting makespan is at least $\frac{2\alpha+1}{\alpha}+2$, yielding the competitive ratio of $\sigma(\efhahome)\ge 2+\frac{1}{2\alpha}$.
\end{proof}

\begin{landscape}
	\def\ps{\phantom{^\star}}
	\begin{table}
		\centering
		\scriptsize
		\begin{tabular}{lcccccccccc}
			\toprule
			                          & \multicolumn{2}{c}{$\alpha = 0.25$}                                   & \multicolumn{2}{c}{$\alpha = 0.5$}            & \multicolumn{2}{c}{$\alpha = 1$}    & \multicolumn{2}{c}{$\alpha = 2$}    & \multicolumn{2}{c}{$\alpha = 4$}                                                                                                                                                                                                                                                    \\
			\cmidrule(l){2-3} \cmidrule(l){4-5} \cmidrule(l){6-7} \cmidrule(l){8-9} \cmidrule(l){10-11}
			                          & $\multitsphome$                                                       & $\singletsphome$                              & $\multitsphome$                     & $\singletsphome$                    & $\multitsphome$                               & $\singletsphome$                              & $\multitsphome$                     & $\singletsphome$                              & $\multitsphome$                               & $\singletsphome$                              \\
			\midrule
			Damage                                                                                                                                                                                                                                                                                                                                                                                                                                                                                                              \\Percentage & \multicolumn{10}{c}{\emph{Influence of damage probability}} \\
			$0.1$                     & $1.22\ps (1.00)^\star$                                                & $\boldsymbol{1.14}\ps (1.03)\ps$              & $1.34^\star (1.20)\ps$              & $\boldsymbol{1.23}\ps (1.09)\ps$    & $1.80\ps (1.19)^\star$                        & $\boldsymbol{1.51}\ps (1.13)\ps$              & $\boldsymbol{1.93}\ps (1.31)\ps$    & $\boldsymbol{2.18}\ps \boldsymbol{(1.72)}\ps$ & $\boldsymbol{1.97}\ps \boldsymbol{(1.46)}\ps$ & $\boldsymbol{4.12}\ps \boldsymbol{(2.87)}\ps$ \\
			\rowcolor{gray!30} $0.3$  & $1.18^\star (1.07)\ps$                                                & $1.09^\star \boldsymbol{(1.04)}\ps$           & $1.40\ps (1.17)^\star$              & $1.19\ps (1.08)\ps$                 & $1.82\ps (1.46)\ps$                           & $1.38\ps \boldsymbol{(1.18)}\ps$              & $1.80\ps \boldsymbol{(1.53)}\ps$    & $1.95\ps (1.24)\ps$                           & $1.43\ps (1.31)\ps$                           & $2.27\ps (1.26)\ps$                           \\
			$0.5$                     & $1.21\ps (1.08)\ps$                                                   & $1.12\ps (1.03)\ps$                           & $1.42\ps \boldsymbol{(1.28)}\ps$    & $1.19\ps \boldsymbol{(1.13)}\ps$    & $\boldsymbol{1.90}\ps \boldsymbol{(1.53)}\ps$ & $1.43\ps (1.17)\ps$                           & $1.81\ps (1.52)\ps$                 & $1.70\ps (1.26)\ps$                           & $1.56\ps (1.30)\ps$                           & $1.74\ps (1.27)\ps$                           \\
			$0.7$                     & $\boldsymbol{1.24}\ps \boldsymbol{(1.14)}\ps$                         & $\boldsymbol{1.14}\ps (1.02)\ps$              & $1.40\ps (1.26)\ps$                 & $1.20\ps (1.06)\ps$                 & $1.80\ps (1.44)\ps$                           & $1.24\ps (1.12)\ps$                           & $1.51\ps (1.36)\ps$                 & $1.28\ps (1.08)\ps$                           & $1.28^\star (1.23)\ps$                        & $1.27\ps (1.13)\ps$                           \\
			$0.9$                     & $1.23\ps (1.13)\ps$                                                   & $\boldsymbol{1.14}\ps (1.00)^\star$           & $\boldsymbol{1.44}\ps (1.23)\ps$    & $1.14^\star (1.04)^\star$           & $1.55^\star (1.38)\ps$                        & $1.20^\star (1.04)^\star$                     & $1.45^\star (1.28)^\star$           & $1.20^\star (1.04)^\star$                     & $1.29\ps (1.20)^\star$                        & $1.20^\star (1.04)^\star$                     \\
			\midrule
			Graph class               & \multicolumn{10}{c}{\emph{Influence of graph class}}                                                                                                                                                                                                                                                                                                                                                                                                                                    \\
			\rowcolor{gray!30} Random & $1.18^\star (1.07)\ps$                                                & $1.09^\star (1.04)^\star$                     & $1.40\ps (1.17)\ps$                 & $1.19^\star (1.08)^\star$           & $1.82\ps (1.46)\ps$                           & $1.38^\star (1.18)\ps$                        & $1.80^\star (1.53)\ps$              & $1.95^\star (1.24)^\star$                     & $1.43^\star (1.31)\ps$                        & $2.27^\star (1.26)^\star$                     \\
			$1$-Center                & $1.20\ps \boldsymbol{(1.11)}\ps$                                      & $1.14\ps (1.06)\ps$                           & $\boldsymbol{1.43}\ps (1.11)\ps$    & $1.34\ps (1.10)\ps$                 & $1.71^\star (1.28)^\star$                     & $\boldsymbol{1.64}\ps \boldsymbol{(1.21)}\ps$ & $\boldsymbol{1.88}\ps (1.37)^\star$ & $\boldsymbol{2.66}\ps \boldsymbol{(1.63)}\ps$ & $1.75\ps (1.28)^\star$                        & $3.31\ps \boldsymbol{(1.87)}\ps$              \\
			$2$-Center                & $1.18^\star (1.01)\ps$                                                & $\boldsymbol{1.23}\ps \boldsymbol{(1.07)}\ps$ & $1.19^\star (1.05)^\star$           & $\boldsymbol{1.36}\ps (1.10)\ps$    & $\boldsymbol{1.91}\ps \boldsymbol{(1.62)}\ps$ & $1.40\ps (1.18)\ps$                           & $1.80^\star (1.52)\ps$              & $2.22\ps (1.40)\ps$                           & $1.48\ps (1.31)\ps$                           & $3.52\ps (1.41)\ps$                           \\
			Coastal                   & $\boldsymbol{1.22}\ps (1.09)\ps$                                      & $1.16\ps (1.05)\ps$                           & $1.41\ps (1.21)\ps$                 & $1.27\ps \boldsymbol{(1.12)}\ps$    & $1.90\ps (1.53)\ps$                           & $1.50\ps (1.20)\ps$                           & $1.80^\star (1.53)\ps$              & $2.06\ps (1.31)\ps$                           & $\boldsymbol{1.81}\ps (1.32)\ps$              & $\boldsymbol{4.09}\ps (1.46)\ps$              \\
			Mountain                  & $1.18^\star (1.00)^\star$                                             & $1.14\ps (1.06)\ps$                           & $1.42\ps \boldsymbol{(1.22)}\ps$    & $1.22\ps (1.08)^\star$              & $1.78\ps (1.48)\ps$                           & $1.44\ps (1.17)^\star$                        & $1.85\ps \boldsymbol{(1.54)}\ps$    & $2.06\ps (1.31)\ps$                           & $1.59\ps \boldsymbol{(1.34)}\ps$              & $2.37\ps (1.35)\ps$                           \\
			\midrule
			Number Nodes              & \multicolumn{10}{c}{\emph{Influence of Number of Nodes in the Graph}}                                                                                                                                                                                                                                                                                                                                                                                                                   \\
			$13$                      & $1.19\ps (1.00)^\star$                                                & $\boldsymbol{1.15}\ps (1.01)^\star$           & $1.41\ps (1.12)^\star$              & $\boldsymbol{1.32}\ps (1.02)^\star$ & $1.88\ps (1.51)\ps$                           & $1.38\ps (1.11)^\star$                        & $1.80\ps \boldsymbol{(1.61)}\ps$    & $1.84\ps \boldsymbol{(1.24)}\ps$              & $1.67\ps (1.34)\ps$                           & $\boldsymbol{2.74}\ps \boldsymbol{(1.36)}\ps$ \\
			$15$                      & $\boldsymbol{1.24}\ps (1.00)^\star$                                   & $1.07^\star (1.01)^\star$                     & $\boldsymbol{1.46}\ps (1.18)\ps$    & $1.14^\star (1.06)\ps$              & $\boldsymbol{1.95}\ps \boldsymbol{(1.55)}\ps$ & $1.37^\star (1.16)\ps$                        & $\boldsymbol{1.92}\ps (1.52)\ps$    & $1.92\ps (1.22)^\star$                        & $1.83\ps \boldsymbol{(1.36)}\ps$              & $2.25\ps (1.34)\ps$                           \\
			\rowcolor{gray!30} $18$   & $1.18^\star (1.07)\ps$                                                & $1.09\ps \boldsymbol{(1.04)}\ps$              & $1.40^\star (1.17)\ps$              & $1.19\ps (1.08)\ps$                 & $1.82\ps (1.46)^\star$                        & $1.38\ps (1.18)\ps$                           & $1.80\ps (1.53)\ps$                 & $\boldsymbol{1.95}\ps \boldsymbol{(1.24)}\ps$ & $1.43^\star (1.31)^\star$                     & $2.27\ps (1.26)\ps$                           \\
			$21$                      & $1.19\ps \boldsymbol{(1.08)}\ps$                                      & $1.14\ps \boldsymbol{(1.04)}\ps$              & $1.40^\star \boldsymbol{(1.24)}\ps$ & $1.22\ps \boldsymbol{(1.12)}\ps$    & $1.66^\star (1.50)\ps$                        & $\boldsymbol{1.49}\ps \boldsymbol{(1.20)}\ps$ & $1.74^\star (1.39)^\star$           & $1.51^\star (1.23)\ps$                        & $\boldsymbol{1.86}\ps (1.32)\ps$              & $2.08^\star (1.24)^\star$                     \\
			\bottomrule
			\multicolumn{11}{l}{\scriptsize
				The \IBASE\ dataset is highlighted in gray.
				For each entry, the first (second) value gives the worst observed (median) competitive ratios $\hat{\sigma}$($\sigma_{\text{med}}$) .
			}                                                                                                                                                                                                                                                                                                                                                                                                                                                                                                                   \\
			\multicolumn{11}{l}{\scriptsize
				For each $\alpha$, the largest (smallest) entries are marked in \textbf{bold} (with $^\star$).
			}                                                                                                                                                                                                                                                                                                                                                                                                                                                                                                                   \\
		\end{tabular}
		\caption{
			\footnotesize{Worst-observed and median competitive ratios for $\multitsphome$ and $\singletsphome$ on \ILARGE\ }
		}
		\label{tab:cr-vs-dp}
	\end{table}

	\begin{table}
		\centering
		\def\ps{\phantom{^\star}}
		\scriptsize
		\begin{tabular}{lccccc}
			\toprule
			Policy           & $\alpha = 0.25$                               & $\alpha = 0.5$                                & $\alpha = 1$                                  & $\alpha = 2$                     & $\alpha = 4$                                  \\
			\midrule
			$\multitsphome$  & $1.24\ps (1.06)\ps$                           & $1.46\ps (1.17)\ps$                           & $1.95\ps (1.49)\ps$                           & $1.92\ps (1.52)\ps$              & $1.86\ps (1.32)\ps$                           \\
			$\singletsphome$ & $1.15^\star (1.03)^\star$                     & $1.32^\star (1.07)^\star$                     & $1.49^\star (1.16)^\star$                     & $1.95\ps (1.23)\ps$              & $2.74\ps (1.27)\ps$                           \\
			$\efhahome$      & $4.82\ps (4.37)\ps$                           & $2.85\ps (2.46)\ps$                           & $1.98\ps (1.60)\ps$                           & $1.76^\star (1.22)^\star$        & $1.77^\star (1.09)^\star$                     \\
			$\efhshome$      & $\boldsymbol{5.44}\ps \boldsymbol{(4.95)}\ps$ & $\boldsymbol{3.40}\ps \boldsymbol{(3.11)}\ps$ & $\boldsymbol{2.60}\ps \boldsymbol{(2.28)}\ps$ & $2.16\ps \boldsymbol{(1.72)}\ps$ & $1.96\ps (1.36)\ps$                           \\
			$\truckonly$     & $1.19\ps (1.06)\ps$                           & $1.35\ps (1.17)\ps$                           & $1.69\ps (1.37)\ps$                           & $\boldsymbol{2.70}\ps (1.47)\ps$ & $\boldsymbol{4.58}\ps \boldsymbol{(1.47)}\ps$ \\
			\bottomrule
			\multicolumn{6}{l}{\scriptsize{
					For each entry, the first (second) value gives the worst observed (median) competitive ratios $\hat{\sigma}$($\sigma_{\text{med}}$).
			}}                                                                                                                                                                                                                                                  \\
			\multicolumn{6}{l}{\scriptsize{
					For each $\alpha$, the largest (smallest) entries are marked in \textbf{bold} (with $^\star$).
			}}                                                                                                                                                                                                                                                  \\
		\end{tabular}
		\caption{
			\footnotesize{Comparison of $\multitsphome$ and $\singletsphome$  to alternative policies on \ILARGE}
		}
		\label{tab:cr-vs-policy}
	\end{table}
\end{landscape}

\subsubsection{Multiple Vehicles} \label{sub:comp-ratio-analysis:multiple}
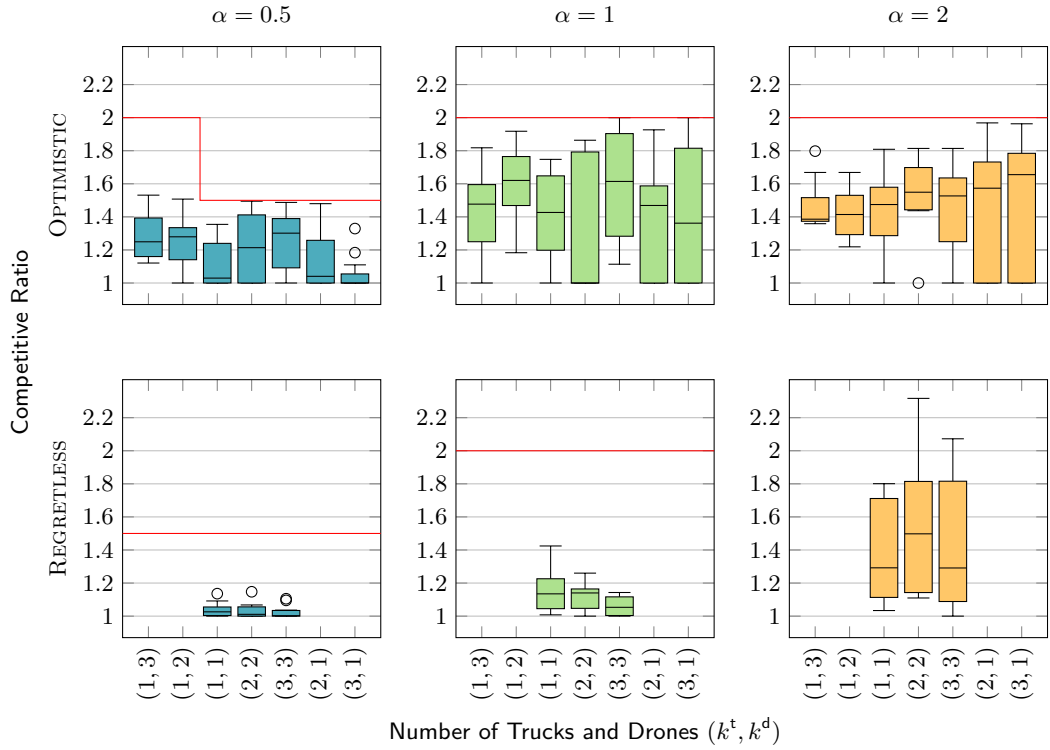
\begin{figure}
	\centering
	\footnotesize
	\begin{tikzpicture}
		\pgfplotstableread[col sep=comma]{ratio_vs_drone_trucks_multitsp.csv}\ratioVStruckdronesMultitsp
		\pgfplotstabletranspose\datatransposed{\ratioVStruckdronesMultitsp}
		\pgfplotstableread[col sep=comma]{ratio_vs_drone_trucks_singletsp.csv}\ratioVStruckdronesSingletsp
		\pgfplotstabletranspose\datatransposedSingle{\ratioVStruckdronesSingletsp}
		\begin{groupplot}[
				group style = {
						group size= 3 by 2,
						ylabels at=edge left,
					},
				width=0.33\textwidth,
				height=0.33\textwidth,
				boxplot/draw direction = y,
				x axis line style = {opacity=1},
				enlarge y limits,
				ymajorgrids,
				xtick = {1,2,3,4,5,6,7},
				xticklabel style = {align=center, font=\footnotesize, rotate=90},
				xticklabels = \empty,
				xmin=0.25,
				xmax=7.75,
				ylabel = {Competitive Ratio},
				ytick = {1, 1.2, 1.4, 1.6, 1.8, 2.0, 2.2},
				ymin = 1,
				ymax = 2.3,
				every boxplot/.style={
						draw=black,
						solid,
						mark=o,
					},
			]
			\nextgroupplot[ylabel={$\multitsphome$}, title={$\alpha = 0.5$}]
			\addplot+[boxplot={draw position=1}, fill=color-a0.5] table[y index=2] {\datatransposed};
			\addplot+[boxplot={draw position=2}, fill=color-a0.5] table[y index=6] {\datatransposed};
			\addplot+[boxplot={draw position=3}, fill=color-a0.5] table[y index=10] {\datatransposed};
			\addplot+[boxplot={draw position=4}, fill=color-a0.5] table[y index=14] {\datatransposed};
			\addplot+[boxplot={draw position=5}, fill=color-a0.5] table[y index=18] {\datatransposed};
			\addplot+[boxplot={draw position=6}, fill=color-a0.5] table[y index=22] {\datatransposed};
			\addplot+[boxplot={draw position=7}, fill=color-a0.5] table[y index=26] {\datatransposed};
			\addplot[draw=red, mark=] coordinates {(0.25, 2) (2.5,2) (2.5, 1.5) (7.75, 1.5)};
			\coordinate (top) at (rel axis cs:0,1);
			\nextgroupplot[title={$\alpha = 1$}]
			\addplot+[boxplot={draw position=1}, fill=color-a1] table[y index=3] {\datatransposed};
			\addplot+[boxplot={draw position=2}, fill=color-a1] table[y index=7] {\datatransposed};
			\addplot+[boxplot={draw position=3}, fill=color-a1] table[y index=11] {\datatransposed};
			\addplot+[boxplot={draw position=4}, fill=color-a1] table[y index=15] {\datatransposed};
			\addplot+[boxplot={draw position=5}, fill=color-a1] table[y index=19] {\datatransposed};
			\addplot+[boxplot={draw position=6}, fill=color-a1] table[y index=23] {\datatransposed};
			\addplot+[boxplot={draw position=7}, fill=color-a1] table[y index=27] {\datatransposed};
			\addplot[draw=red, mark=] coordinates {(0.25, 2) (7.75, 2)};
			\nextgroupplot[title={$\alpha = 2$}]
			\addplot+[boxplot={draw position=1}, fill=color-a2] table[y index=4] {\datatransposed};
			\addplot+[boxplot={draw position=2}, fill=color-a2] table[y index=8] {\datatransposed};
			\addplot+[boxplot={draw position=3}, fill=color-a2] table[y index=12] {\datatransposed};
			\addplot+[boxplot={draw position=4}, fill=color-a2] table[y index=16] {\datatransposed};
			\addplot+[boxplot={draw position=5}, fill=color-a2] table[y index=20] {\datatransposed};
			\addplot+[boxplot={draw position=6}, fill=color-a2] table[y index=24] {\datatransposed};
			\addplot+[boxplot={draw position=7}, fill=color-a2] table[y index=28] {\datatransposed};
			\addplot[draw=red, mark=] coordinates {(0.25, 2) (7.75, 2)};
			\nextgroupplot[ylabel={$\singletsphome$}, xticklabels = {{$(1, 3)$}, {$(1, 2)$}, {$(1, 1)$}, {$(2, 2)$}, {$(3, 3)$}, {$(2, 1)$}, {$(3, 1)$}}]
			\addplot+[boxplot={draw position=3}, fill=color-a0.5] table[y index=2] {\datatransposedSingle};
			\addplot+[boxplot={draw position=4}, fill=color-a0.5] table[y index=6] {\datatransposedSingle};
			\addplot+[boxplot={draw position=5}, fill=color-a0.5] table[y index=10] {\datatransposedSingle};
			\addplot[draw=red, mark=] coordinates {(0.25, 1.5) (7.75, 1.5)};
			\nextgroupplot[xticklabels = {{$(1, 3)$}, {$(1, 2)$}, {$(1, 1)$}, {$(2, 2)$}, {$(3, 3)$}, {$(2, 1)$}, {$(3, 1)$}}, xlabel = {Number of Trucks and Drones  $(\nrtrucks, \nrdrones)$}]
			\addplot+[boxplot={draw position=3}, fill=color-a1] table[y index=3] {\datatransposedSingle};
			\addplot+[boxplot={draw position=4}, fill=color-a1] table[y index=7] {\datatransposedSingle};
			\addplot+[boxplot={draw position=5}, fill=color-a1] table[y index=11] {\datatransposedSingle};
			\addplot[draw=red, mark=] coordinates {(0.25, 2) (7.75, 2)};
			\nextgroupplot[xticklabels = {{$(1, 3)$}, {$(1, 2)$}, {$(1, 1)$}, {$(2, 2)$}, {$(3, 3)$}, {$(2, 1)$}, {$(3, 1)$}}]
			\addplot+[boxplot={draw position=3}, fill=color-a2] table[y index=4] {\datatransposedSingle};
			\addplot+[boxplot={draw position=4}, fill=color-a2] table[y index=8] {\datatransposedSingle};
			\addplot+[boxplot={draw position=5}, fill=color-a2] table[y index=12] {\datatransposedSingle};
			\addplot[draw=red, mark=] coordinates {(0.25, 3) (7.75, 3)};
			\coordinate (bot) at (rel axis cs:1,0);
		\end{groupplot}
		\path (top-|current bounding box.west)--
		node[anchor=south,rotate=90] {Competitive Ratio}
		(bot-|current bounding box.west);
	\end{tikzpicture}
	\caption{
		\footnotesize{Competitive ratios $\hat{\sigma}$ for $\multitsphome$ and $\singletsphome$ for several trucks and drones on \ISMALL\ \\
			\textit{Note.} The red line indicates analytical competitive ratios $\sigma$ from Section~\ref{sec:analysis}.
			The red line  for $\alpha=2$ for \singletsphome\ depicts value  $3$ and is not visible in the figure.}
	}
	\label{fig:multiplevehicles_multi_tsp}
\end{figure}
As shown in Figure~\ref{fig:multiplevehicles_multi_tsp}, \singletsphome\ outperforms \multitsphome\ for small values of $\alpha$ across all  examined settings. However, for $\alpha=2$, even in this relatively small instance set, some instances exhibit a competitive ratio for \singletsphome\ exceeding 2, which is the competitive ratio for \multitsphome.
Overall, the worst-observed competitive ratios in Figure~\ref{fig:multiplevehicles_multi_tsp} lie further from the theoretical worst-case bounds, which may reflect the limited size of the tested instance set rather than inherent algorithmic behavior.

\subsection{Drone-Impact Ratio Analysis} \label{sub:risk-analysis}
Figure~\ref{fig:risk} reports the results on the observed drone-impact ratio for \ILARGE\ instances with $\nrtrucks = \nrdrones = 1$.
For $\multitsphome$, we observe for all $\alpha$ solutions with a drone-impact ratio exceeding 1. $\singletsphome$ outperforms $\multitsphome$ both in terms of the worst-observed and, for $\alpha \leq 2$, median drone-impact ratio.
In contrast, with respect to the best observed drone-impact ratio, $\multitsphome$ outperforms $\singletsphome$ for all values of $\alpha$.
Overall, the analytical ratios derived in Section~\ref{sec:analysis} closely predict the observed performance. In particular, the worst observed drone-impact ratios for both policies (except for $\alpha = 1$ in $\multitsphome$), as well as the best observed ratios for $\multitsphome$, align with the theoretical bounds.

\begin{figure}
	\centering
	\footnotesize
	\begin{tikzpicture}
		\pgfplotstableread[col sep=comma]{risk_vs_alpha_multitsp.csv}\riskVSalphaMultitsp
		\pgfplotstabletranspose\datatransposed{\riskVSalphaMultitsp}
		\pgfplotstableread[col sep=comma]{risk_vs_alpha_singletsp.csv}\riskVSalphaSingletsp
		\pgfplotstabletranspose\datatransposedSingle{\riskVSalphaSingletsp}
		\begin{groupplot}[
				group style = {
						group size= 2 by 1,
						ylabels at=edge left,
					},
				width=0.5\textwidth,
				height=0.33\textwidth,
				boxplot/draw direction = y,
				boxplot/box extend=0.2,
				x axis line style = {opacity=1},
				enlarge y limits,
				ymajorgrids,
				xtick = {0.25, 1, 2, 4},
				xticklabel style = {align=center, font=\footnotesize},
				extra x ticks = {0.5},
				extra x tick style={xticklabel style={yshift=-10pt, font=\small}},
				xlabel = {$\alpha$},
				xmin = 0,
				xmax = 4.75,
				ylabel = {Drone-Impact Ratio},
				ytick = {0.2, 0.4, 0.6, 0.8, 1, 1.2, 1.4, 1.6},
				ymin = 0.2,
				ymax = 1.6,
				every boxplot/.style={
						draw=black,
						solid,
						mark=o,
					},
			]
			\nextgroupplot[title={$\multitsphome$}]
			\addplot+[boxplot={draw position=0.25}, fill=color-a0.25] table[y index=1] {\datatransposed};
			\addplot+[boxplot={draw position=0.5}, fill=color-a0.5] table[y index=2] {\datatransposed};
			\addplot+[boxplot={draw position=1}, fill=color-a1] table[y index=3] {\datatransposed};
			\addplot+[boxplot={draw position=2}, fill=color-a2] table[y index=4] {\datatransposed};
			\addplot+[boxplot={draw position=4}, fill=color-a4] table[y index=5] {\datatransposed};
			\addplot[draw=red, mark=] coordinates {(0, 1) (1, 2)};
			\draw[smooth, domain=1:5, red, variable=\x] plot ({\x, 1/\x + 1});
			\draw[smooth, domain=0:5, variable=\x, red] plot ({\x, 1 / (1+\x)});
			\nextgroupplot[title={$\singletsphome$}]
			\addplot+[boxplot={draw position=0.25}, fill=color-a0.25] table[y index=1] {\datatransposedSingle};
			\addplot+[boxplot={draw position=0.5}, fill=color-a0.5] table[y index=2] {\datatransposedSingle};
			\addplot+[boxplot={draw position=1}, fill=color-a1] table[y index=3] {\datatransposedSingle};
			\addplot+[boxplot={draw position=2}, fill=color-a2] table[y index=4] {\datatransposedSingle};
			\addplot+[boxplot={draw position=4}, fill=color-a4] table[y index=5] {\datatransposedSingle};
			\addplot[draw=red, mark=] coordinates {(0, 1) (4.75, 1)};
			\draw[smooth, domain=0:5, variable=\x, red] plot ({\x, 1 / (1+\x)});
		\end{groupplot}
	\end{tikzpicture}
	\caption{
		\footnotesize{Observed drone-impact ratios for $\multitsphome$ and $\singletsphome$ on \ILARGE\ \\
			\textit{Note.} The red lines indicates analytical drone-impact ratios $\bar{\omega}$ and $\underline{\omega}$ from Section~\ref{sec:analysis}.}
	}
	\label{fig:risk}
\end{figure}
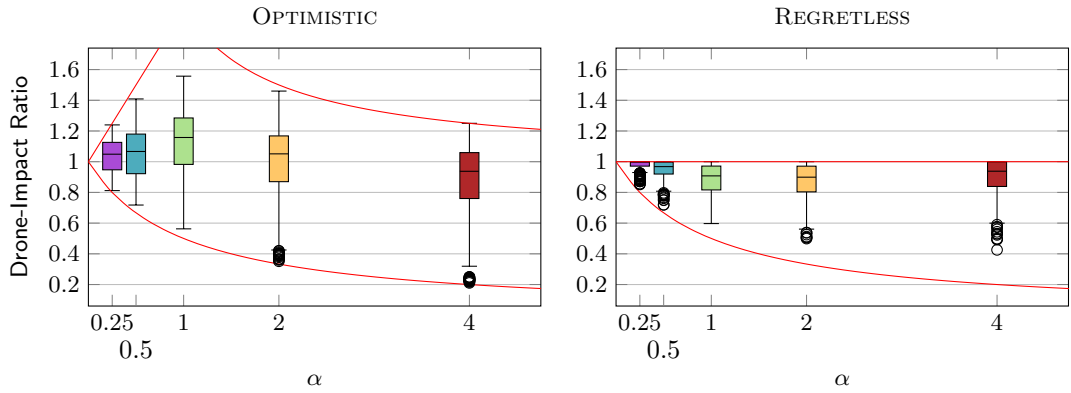

\section{Conclusions and future research directions}\label{sec:conclusion}

Timely relief after disasters is of crucial importance. Surveillance drones can support relief operations by providing information on the actual demand for relief supplies at potential demand locations.
However, the efficient operation of a fleet of relief distribution trucks and surveillance drones under uncertainty remains challenging:  In this paper we are able to show some general lower bounds on both the competitive ratio and drone-impact ratios, which demonstrates theoretical limitations on the performance of any deterministic online policy. Our results highlight a fundamental trade-off: an online policy that is optimal with respect to the competitive ratio need not be \emph{regretless}, even when drones are fast. In particular, the use of drones may lead to substantially worse performance compared to truck-only operations. This counterintuitive phenomenon arises from  a nontrivial trade-off between exploration and delivery efficiency.

Motivated by this observation, we study two natural policies for RDP: $\multitsphome$, which achieves the best possible competitive ratio, and $\singletsphome$, which is regretless and achieves the best possible  drone-impact ratios $\bar{\omega}$ and $\underline{\omega}$. On the positive side, we can upper-bound the worst-case performance of both investigated policies \multitsphome\ and \singletsphome. Our extensive computational experiments -- including on instances that mimic mountain and coastal regions, which are common in disaster relief operations -- show that the theoretically derived bounds \textit{closely approximate the observed performance} of the policies and, thus, provide invaluable insights for disaster relief managers.
Based on theoretical results and supported by the experimental evaluation, we conclude that \singletsphome\ performs better when drones are slow in comparison to trucks, and \multitsphome\ when drones are faster than trucks.
Several directions for future research remain. First, improved bounds on the examined ratios  may be attainable in a few settings where the current bounds are not tight. Overall, it remains an open question whether an online policy can simultaneously achieve best possible values for all three ratios: the competitive ratio and the two drone-impact ratios.
Second, alternative policies merit further investigation, especially those based on immediate reoptimization, such as \efhahome, which performs well in our experiments with fast drones. In addition, other objective functions should be considered, such as a \emph{nomadic} variant in which the makespan is defined by the completion time of the last delivery rather than the return of the final vehicle to the depot.

More broadly, solving RDP, as a first stylized model for drone-supported disaster relief distribution, is but a first step towards understanding the full complexity of the real-world relief distribution operations under uncertainty.  Future work could extend the model to incorporate demand quantities and vehicle capacities, travel range limits, fairness considerations, and multiple sources of uncertainty, including road damages and stochastic travel times. Another promising direction is the integration of imperfect probabilistic information on damages into the decision process. Such extensions may ultimately support the development of more comprehensive and practically relevant planning approaches for disaster relief operations under uncertainty.

\singlespacing
\bibliography{mybib}

\newpage
\onehalfspacing

\renewcommand{\thesection}{\Alph{section}}
\setcounter{section}{0}
\pagenumbering{roman}
\setcounter{page}{1}

\part*{Appendix}
\section{Proofs omitted in main text}
\subsection{Proofs on basic technical results}

\lemmabasic*
\begin{proof}\label{proof:lemma-basic}
	Relations~\eqref{eq:moretrucks}-\eqref{eq:moredrones} are straightforward, since, \textit{ceteris paribus}, increasing the number of vehicles cannot worsen the makespan objective. In \eqref{eq:redistribution_to_trucks} and \eqref{eq:redistribution_to_drones}, we evenly reallocate the optimal tours of $\tsp_{m,n}(W)$ to $m'\geq 1$ trucks or $n'\geq 1$ drones, respectively.
\end{proof}

\lemboundopt*
\begin{proof}\label{proof:lem-boundOPT}
	This follows directly from the fact that each node needs to be visited by at least one vehicle, and that each damaged node needs to be visited by a truck.
\end{proof}

\lemmarelationsgraphs*
\begin{proof} \label{proof:lemma:relations_graphs}
	Relations~\eqref{eq:one_truck_star} and \eqref{eq:one_drone_star} follow immediately from the definition.
	In \eqref{eq:truck_only_star}, the minimum makespan is received by distributing nodes among truck as evenly as possible, in this case -- by the pigeonhole principle -- at least one truck travels over $\left\lceil\frac{n}{\nrtrucks}\right\rceil$ nondepot nodes. A similar argument holds for \eqref{eq:drone_only_star}.

	Recall that transitions between non-depot nodes must pass through the depot. Moreover, by the triangle inequality, it suffices to consider tours that visit each node in $W'_1\cup W'_2, n'=|W'_1|+|W'_2|,$ exacly once while solving \tspmnname.  To prove \eqref{eq:one_truck_star2} and \eqref{eq:one_drone_star2}, consider the total length of any such tour over nodes of $W'_1\cup W'_2$: $c(\depot, v_{i_1}, \depot, v_{i_2}, \depot, \ldots, \depot, v_{n'}, \depot)=\sum_{v\in W'_1\cup W'_2}c(\depot, v, \depot)=2d_1|W'_1|+2d_2|W'_2|$. The resulting expression is independent of the visiting order of nodes, which completes the proof.
\end{proof}

\lemmaupperboundimegabestcasesingletsp*
\begin{proof} \label{proof:lemma:upperbound_omega_bestcase_singletsp}
	Consider the same example as in Lemma~\ref{lemma:upperbound_omega_bestcase_multitsp} with $\tsp_{\nrtrucks,0}(C)= 2\nrtrucks+2\alpha \nrdrones$.

	Consider the case of $\alpha\neq 1$, as the case of $\alpha=1$ is straightforward because the nodes in $W_1$ and $W_2$ are indistinguishable.
	In $\singletsphome$, each truck $i$ receives $\frac{\nrdrones}{\nrtrucks}$ drones to assist in its tour $s^*_i$ consisting of $k^t$ nodes in $W_1$ and $k^d$ nodes in $W_2$.  For each tour $s^*_i$, its partitioning into segments is equivalent to solving \tspmnname\ over graph $\tilde{\mathcal{G}}(\nrtrucks, \nrdrones,1, \alpha)$ formed by the nodes of $s^*_i$ with 1 truck and $\frac{k^d}{k^t}$ drones. By \eqref{eq:mtsp_two_levelstar} and $f=k^t$, the resulting makespan is $2f=2k^t$ and, by the pigeonhole principle given $\alpha\neq 1$, the unique optimal solution is to assign all $k^t$ level-1 nodes to the truck and equally distribute $n^d$ level-2 nodes among the drones of this truck. At time $\theta=2k^t-1$, when the truck arrives to the last node of its segment, all the nodes in the drone segments have just been explored by the drones; so that the truck and the drones directly return to the depot resulting in $\singletsphome(I)=2k^t.$

	It follows $\underline{\omega}(\singletsphome)\leq\frac{\singletsphome(I)}{\tsp_{\nrtrucks,0}(C)}= \frac{1}{1+\alpha\frac{\nrdrones}{\nrtrucks}}$.
\end{proof}

\subsection{Additional proofs for bounds on drone-impact ratio and competitive ratio}

\lemmalowerboundomegaworstcasemultitspnew*
\begin{proof} \label{proof:lemma:lowerbound_omega_worstcase_multitsp_new}
	For $\alpha>1$, we have $\min\{1+\frac{1}{\alpha}, \frac{1+\alpha}{1+\epsilon}\}=1+\frac{1}{\alpha}$.
	Consider an instance $I$ with all nodes damaged ($D=C$) and the depot node $v_0$ together with nodes $C$ forming a star graph $\tilde{\mathcal{G}}(\min\{\nrtrucks,\nrdrones\},1)$.
	From \eqref{eq:truck_only_star}, $\tsp_{\nrtrucks,0}(C)=2$.
	$\multitsphome$ performs only drone tours in the \initial\ step, with makespan (using \eqref{eq:drone_only_star}) of $\frac{2}{\alpha}$.
	Since all these nodes turn damaged, trucks have to revisit these nodes resulting in the makespan of 2 for this second step of $\multitsphome$. Hence, $\multitsphome(I)=\frac{2}{\alpha}+2$ and $\frac{\multitsphome(I)}{\tsp_{\nrtrucks,0}(C)}= 1+\frac{1}{\alpha}$.

	For $\alpha\le 1$ we have $\min\{1+\frac{1}{\alpha}, \frac{1+\alpha}{1+\epsilon}\}=\frac{1+\alpha}{1+\epsilon}$.
	Consider an instance $I$ with all nodes damaged ($D=C$). The set of non-depot nodes $C=P\cup F$
	consists of \emph{proximate} nodes $P$ and \emph{far-away} nodes $F$ with $|P|=k^t$ and $|F|=\min\{k^t,k^d\}$.
	For each $p\in P$ we determine one (distinct) \emph{neighbor} node $f(p)\in F$.
	Our edge set is $E=\{(\depot,c):c\in C\}\cup \{(p,f(p))~: \forall p \in P\}$ with $c(\depot,p)=1~ \forall p\in P$, $c(\depot,f)=\alpha~ \forall f\in F$, $c(p,f)=\alpha-1+\varepsilon~ \forall p\in P$.

	Then, $\tsp_{\nrtrucks,0}(C)= 2+2\epsilon$, as each truck visits one far-away node and, if available, its neighbor proximate node.

	In the first step of $\multitsphome$,  each vehicle visits at most one node with the makespan of $\tsp_{\nrtrucks,\nrdrones}(C)= 2$. Since all these nodes are damaged and at least $\min\{k^t,k^d\}$ nodes have not been visited by trucks, trucks have to revisit these nodes, resulting in the makespan of $2\alpha$ for this second step of $\multitsphome$. Hence, $\multitsphome(I)=2+2\alpha$ and $\frac{\multitsphome(I)}{\tsp_{\nrtrucks,0}(C)}= \frac{1+\alpha}{1+\epsilon}$.
\end{proof}

\lemmacorupperbounds*
\begin{proof} \label{proof:cor-upperbounds}
	On the one hand, from Lemma~\ref{lem-boundOPT} and Lemma~\ref{lem-boundMultiTSP}, it follows that
	\[\sigma(\multitsphome)
		=\sup_I \frac{\multitsphome(I)}{\OPThome(I)}
		\leq \frac{\tsp_{\nrtrucks,\nrdrones}(C) + \tsp_{\nrtrucks,0}(D)}
		{ \max \{\tsp_{\nrtrucks,\nrdrones}(C), \tsp_{\nrtrucks,0}(D)\}}
		\leq 2.\]

	On the other hand, from Lemmata~\ref{lem-boundOPT} and \ref{lem-boundMultiTSP} and  from \eqref{eq:morenodes},   \eqref{eq:redistribution_to_trucks}, it follows
	\begin{align*}
		\sigma(\multitsphome) =\sup_I \frac{\multitsphome(I)}{\OPThome(I)}
		\leq \frac{\tsp_{\nrtrucks,\nrdrones}(C) + \tsp_{\nrtrucks,0}(\Ddrone)}
		{ \max \{\tsp_{\nrtrucks,\nrdrones}(C), \tsp_{\nrtrucks,0}(D)\}}\leq \\
		\le  \frac{\tsp_{\nrtrucks,\nrdrones}(C) +\alpha \left\lceil \frac{\nrdrones}{\nrtrucks}\right\rceil \tsp_{0,\nrdrones}(\Ddrone)}
		{ \max \{\tsp_{\nrtrucks,\nrdrones}(C), \tsp_{\nrtrucks,0}(D)\}}
		\le  \frac{\tsp_{\nrtrucks,\nrdrones}(C) +\alpha \left\lceil \frac{\nrdrones}{\nrtrucks}\right\rceil \tsp_{0,\nrdrones}(C)}
		{ \tsp_{\nrtrucks,\nrdrones}(C)} \leq                                \\
		\le  1 +\alpha \left\lceil \frac{\nrdrones}{\nrtrucks}\right\rceil.
	\end{align*}
\end{proof}

\lemmacorupperboundssingletsp*
\begin{proof} \label{proof:cor-upperbounds_singletsp}
	From Lemmata~\ref{lem-boundOPT} and \ref{lem-singletsphome} and from \eqref{eq:redistribution_to_trucks}, it follows
	\begin{align*}
		\sigma(\singletsphome) =\sup_I \frac{\singletsphome(I)}{\OPThome(I)}
		\leq \frac{\tsp_{\nrtrucks,0}(C)}
		{ \max \{\tsp_{\nrtrucks,\nrdrones}(C), \tsp_{\nrtrucks,0}(D)\}}\leq \\
		\le  \frac{\left(1+\alpha \left\lceil \frac{\nrdrones}{\nrtrucks}\right\rceil\right) \tsp_{\nrtrucks,\nrdrones}(C)}
		{ \tsp_{\nrtrucks,\nrdrones}(C)} \leq 1 +\alpha \left\lceil \frac{\nrdrones}{\nrtrucks}\right\rceil.
	\end{align*}
\end{proof}

\lemmakdkvlbnew*
\begin{proof} \label{proof:lemma:kdkv_LB_2new2}
	\begin{figure}
		\centering
		\begin{subfigure}[b]{0.495\textwidth}
			\includegraphics[scale=0.45]{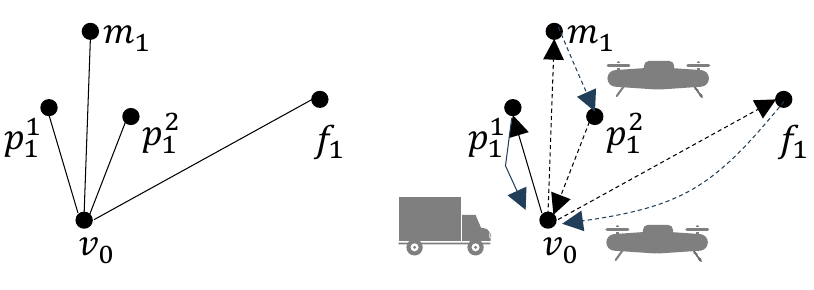}
			\caption{\footnotesize{Case of one slow truck and two fast drones ($\alpha>1$)
				}}
			\label{fig:Figure1}
		\end{subfigure}
		\begin{subfigure}[b]{0.495\textwidth}
			\includegraphics[scale=0.5]{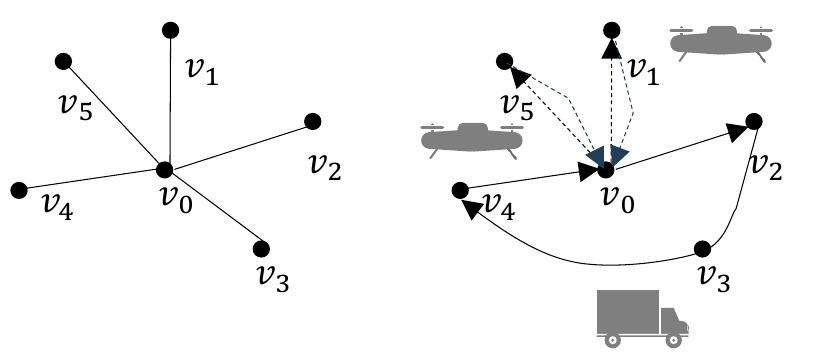}
			\caption{\footnotesize{Case of one truck, two drones, and $\alpha=\frac{1}{3}$
				}}
			\label{fig:worst_example_2}
		\end{subfigure}
		\caption{Worst case instances and initial tours of $\multitsphome$}\label{fig:worst_examples}
	\end{figure}

	In the following, we distinguish between \emph{fast} and \emph{slow} vehicles: \emph{fast} refers to drones when $\alpha > 1$ and to trucks when $\alpha < 1$, while \emph{slow} refers to the other vehicle type. Let $k^{\text{fast}}, k^{\text{slow}} \in \{\nrtrucks, \nrdrones\}$ denote the number of fast and slow vehicles, respectively, and set $k^{\min}=\min\{k^{\text{fast}}, k^{\text{slow}}\}$.
	Consider an instance $I$ defined on a graph $G = (V, E, c)$, where set $C=P\cup M\cup F$ is partitioned into  $k^{\text{slow}}+k^{\text{min}}$  proximate nodes $P$, $k^{\text{min}}$ medium-distance nodes $M$, and $k^{\text{fast}}-k^{\text{min}}$ far-distance nodes $F$.
	The edge set consists of edges $E_C=\{\{\depot, v\}:v\in C\}$, which connect the nodes to the depot and $2|M|$ additional edges connect each node in $M$ to exactly two (pairwise distinct) nodes in $P$, i.e.,  $E_{PM}=\{\{m,p^1_m\}, \{m,p^2_m\}: m\in M\}$ with $p^i_m\in P$ and $p^i_m\neq p^{i'}_{m'}$ if $i\neq i'$ or $m\neq m'$.
	For a given $T>0$, for the proximate nodes $p\in P$ we define the edge labels $c(\depot,p)$ such that
	\emph{proximate nodes} are located equidistantly from the depot,
	and that the makespan of a tour
	of a slow vehicle to any of these nodes equals $T$: $c^{\text{slow}}(\depot,p, \depot)=T$, $\forall p\in P$.
	Analogously, the edge labels for the \emph{far-distance nodes} are chosen such that these nodes are located equidistantly at the largest distance from the depot, such that the makespan of a tour of a fast vehicle to any of these nodes equals $T$: $c^{\text{fast}}(\depot,f, \depot)=T, \forall f\in F$.
	The edge labels $c(\depot,m):=a$ and $c(m,p_i^m):=b$ for all $m\in M$ and $p_i^m \in P$ are chosen such that
	$c^{\text{slow}}(\depot,m, \depot)>T$ and $c^{\text{fast}}(\depot,p^1_m, m, \depot)=c^{\text{fast}}(\depot,p^2_i, m, \depot)=T$.

	The set of damaged nodes consists of one proximate node: $D\in P, |D|=1$.

	Figure~\ref{fig:Figure1} illustrates an example of two fast drones ($\alpha>1$) and one truck. The constructed instance has $|C|=4$ nodes besides the depot: a far-distance node $f_1$, a medium-distance node $m_1$ and the associated proximate nodes $p^1_1$ and $p^2_1$. The figure also depict the tours of the truck and the drones in the\initial-step
	solution of $\multitsphome$.

	In an optimal \tspmnname\ solution on the instance $(\nrvehicles, \nrtrucks, \alpha, G)$,  $|F|$ fast vehicles visit exactly one node in $F$ and return to the depot, the remaining $k^{\text{fast}}-|F|$
	fast vehicles perform a tour of type $(\depot,p^1_i, m_i, \depot)$, visiting a node in $M$ and its associated proximate node in $P$ before returning to the depot.
	Each slow vehicle visits exactly one  node in $P$ and returns to the depot.
	Since the set of damaged nodes consists of exactly one proximate node, we have $\OPT(I)=\tsp_{\nrtrucks,\nrdrones}(C)=T$.

	In \multitsphome\ we complete the  sketched \tspmnname\ solution in the \initial\ step in time $\tsp_{\nrtrucks,\nrdrones}(C)=T$, and if the drone visited the damaged node, the \secondary\ step  of $\multitsphome$ takes at least $c^t(\depot, v, \depot)$ time units, which equals $T$ if trucks are slow and $\alpha T$ if trucks are fast.
	Hence $\frac{\multitsphome(I)}{\OPT(I^*)}= \min\{2,1+\alpha\}$.
\end{proof}

\lemmawcbigalphaoptimistic*
\begin{proof} \label{proof:lemma:wc-bigalpha1optimistic}
	Observe that the lemma requires $\left \lceil \frac{\nrdrones}{\nrtrucks}\right \rceil\ge \frac{1}{\alpha} $.
	Consider an instance, in which the depot node $\depot$ and nodes $C$ form a star graph $\tilde{\mathcal{G}}(\nrdrones + \left\lceil \frac{1}{\alpha}\right \rceil \nrtrucks,1)$.

	We first argue that on this graph, $\tsp_{\nrtrucks,\nrdrones}(C)=2\left\lceil \frac{1}{\alpha}\right \rceil$: this is attained if each drone visits one non-depot node and each truck visits $\left\lceil \frac{1}{\alpha}\right \rceil$ non-depot nodes. Indeed, if a drone would visit two nodes or more, its tour duration would be at least $\frac{4}{\alpha}> \frac{2(1+\alpha)}{\alpha}= 2(\frac{1}{\alpha}+1)>2 \left\lceil \frac{1}{\alpha}\right \rceil$.

	Regardless of which nodes are visited by drones in $\multitsphome$, we can always construct an instance, in which
	$n'=\min\{\nrdrones, \left\lceil \frac{1}{\alpha}\right \rceil\nrtrucks\}$
	of these nodes are damaged. Consequently, in \multitsphome, these nodes must be revisited by the $\nrtrucks$ trucks.

	If $n'=\left\lceil \frac{1}{\alpha}\right \rceil\nrtrucks$, then each truck visits $\left\lceil \frac{1}{\alpha}\right\rceil$ non-depot nodes in the second step of $\multitsphome$. If $n'=\nrdrones$, then there is at least one truck that visits $\left \lceil \frac{\nrdrones}{\nrtrucks}\right \rceil\ge \frac{1}{\alpha}$ non-depot
	nodes, which we round up since this number is integer. Hence, there is a truck that visits at least $\left\lceil \frac{1}{\alpha}\right \rceil$
	non-depot nodes and the second step of $\multitsphome$ takes at least $2\left\lceil \frac{1}{\alpha}\right \rceil$.

	Under the full information, $OPT(I^*)=\tsp_{\nrtrucks,\nrdrones}(C)=2\left\lceil \frac{1}{\alpha}\right \rceil$, with all $n'$ damaged nodes visited by $\nrtrucks$ trucks. Hence,  $\frac{\multitsphome(I)}{\OPT(I^*)}=2$ follows.

\end{proof}

\lammawcratiosmaller*
\begin{proof} \label{proof:lemma:wc-ratiosmaller2new}
	Given $\alpha\left\lceil \frac{\nrdrones}{\nrtrucks}\right \rceil<1$, follows $\alpha<1$. Consider an instance, in which the depot node $\depot$ and nodes $C$ form a graph $\tilde{\mathcal{G}}(\nrdrones + \left\lfloor \frac{1}{\alpha}\right \rfloor \nrtrucks,1)$.

	We first argue that on this graph, $\tsp_{\nrtrucks,\nrdrones}=\frac{2}{\alpha}$: this is attained if each drone visits one node, and each truck visits $\left\lfloor \frac{1}{\alpha}\right \rfloor$ nodes. By \eqref{eq:one_truck_star} and \eqref{eq:one_drone_star}, any solution   where a truck visits more than $\left\lfloor \frac{1}{\alpha}\right \rfloor$ nodes or any drone visits more than one node, has a larger makespan.

	Now assume that $n'=\min\{\left\lfloor \frac{1}{\alpha}\right \rfloor \nrtrucks,\nrdrones\}$ nodes that have been visited  by a drone in \multitsphome\ are damaged.
	Then  these nodes have to be revisited by the trucks with at least one truck visiting $\left\lceil \frac{n'}{\nrtrucks}\right \rceil$ nodes in the \secondary\ step. The makespan of this step is $2 \left\lceil \frac{n'}{\nrtrucks}\right \rceil$ and $\multitsphome(I)=\frac{2}{\alpha}+2 \left\lceil \frac{n'}{\nrtrucks}\right \rceil$ for this instance $I$.

	In the optimal $\text{RDP}^*$ solution, the $n'$ are served by the truck tours, and $OPT(I^*)=\tsp_{\nrtrucks,\nrdrones}(I^*)=\frac{2}{\alpha}$.

	Finally, observe that $n'=k^d$, since
	\begin{align*}
		\alpha \cdot \left\lceil \frac{\nrdrones}{\nrtrucks}\right \rceil<1
		\Rightarrow \left\lceil \frac{\nrdrones}{\nrtrucks}\right \rceil <\frac{1}{\alpha}
		\Rightarrow
		\left\lceil \frac{\nrdrones}{\nrtrucks}\right \rceil \le \left\lfloor \frac{1}{\alpha}\right\rfloor
		\Rightarrow
		\frac{\nrdrones}{\nrtrucks} \le \left\lfloor \frac{1}{\alpha}\right\rfloor
		\Rightarrow \nrdrones  \le \left\lfloor \frac{1}{\alpha}\right\rfloor \nrtrucks
	\end{align*}
	We obtain $\frac{\multitsphome(I)}{\OPT(I^*)} = 1+\left \lceil \frac{\nrdrones}{\nrtrucks}\right \rceil\alpha$.

	Figure~\ref{fig:worst_example_2} illustrates $I$ for $k^t=1, k^d=2,$ and $\alpha=\frac{1}{3}.$
\end{proof}

\lemmawcbigalpha*
\begin{proof} \label{proof:lemma:wc-bigalpha1-1}
	Consider an instance $I$ with depot node $v_0$ and nodes $C$ forming a star graph $\gstar(b\nrtrucks+ \nrdrones,1)$.
	Observe that if the number of damaged nodes $|D|\leq bk^t$, then $OPT(I^*)=\tsp_{\nrtrucks,\nrdrones}(C)=2b$. Indeed, each truck visits at most $b$ non-depot nodes and each drones visits at most 1 non-depot node in solutions with the same or better objective value; given the total number of nodes, these relations hold exactly and the makespan follows from \eqref{eq:one_truck_star}-\eqref{eq:one_drone_star}.

	In the \initial\ phase of $\singletsphome$, truck tours that visit all nodes  with the makespan $\tsp_{\nrtrucks,0}(C)=2(b+\left\lceil\frac{\nrdrones}{\nrtrucks}\right\rceil)$ are built. There is at least one tour with exactly $(b+\left\lceil\frac{\nrdrones}{\nrtrucks}\right\rceil)$ non-depot nodes.
	Consider an instance $I$, in which all but the first non-depot node
	initially assigned to one such truck (and its supporting drones) are damaged. This is possible if $\left\lceil\frac{\nrdrones}{\nrtrucks}\right\rceil+b-1\leq k^t$. Then, $\singletsphome= 2(b+\left\lceil\frac{\nrdrones}{\nrtrucks}\right\rceil)$, and $\frac{\singletsphome(I)}{\OPT(I^*)}=1+\alpha \left\lceil \frac{\nrdrones}{\nrtrucks}\right\rceil$.
\end{proof}

\lemmawcbigalphasingletsp*
\begin{proof} \label{proof:lemma:wc-bigalpha1part2-singletsp}
	Consider an instance $I$ with depot node $v_0$ and nodes $C$ forming a star graph $\gstar(\nrtrucks+ \alpha\nrdrones,1)$.
	Observe that if the number of damaged nodes $|D|\leq k^t$, then $OPT(I^*)=\tsp_{\nrtrucks,\nrdrones}(C)=2$. Indeed, each truck visits at most 1 non-depot node and each drones visits at most $\alpha$ non-depot nodes in solutions with the same or better objective value; given the total number of nodes, these relations hold exactly and the makespan follows from \eqref{eq:one_truck_star}-\eqref{eq:one_drone_star}.

	In the \initial\ phase of $\singletsphome$, truck tours that visit all nodes  with the makespan $\tsp_{\nrtrucks,0}(C)=2(1+\left\lceil\frac{\alpha\nrdrones}{\nrtrucks}\right\rceil)$ are built. There is at least one tour with exactly $(1+\left\lceil\frac{\alpha\nrdrones}{\nrtrucks}\right\rceil)$ non-depot nodes.
	Consider an instance $I$, in which all but the first non-depot node initially assigned to one such truck (and its supporting drones) are damaged. This is possible if $1+\left\lceil\frac{\alpha\nrdrones}{\nrtrucks}\right\rceil-1\leq k^t$. Then, $\singletsphome= 2(1+\left\lceil\frac{\alpha\nrdrones}{\nrtrucks}\right\rceil)$ and $\frac{\singletsphome(I)}{\OPT(I^*)}=1+\left\lceil \frac{\alpha \nrdrones}{\nrtrucks}\right\rceil$.
\end{proof}

\lemalphainn*
\begin{proof} \label{proof:lem-alphainN}
	Consider an instance $I$ with $|D|\leq \nrtrucks$ and the depot $\depot$ together with non-depot nodes $C$ forming a star graph $\gstar(\alpha \nrdrones+\nrtrucks,1)$.
	The optimal $\text{RDP}^*$ solution has a makespan of $2$: the trucks visit the damaged nodes and the drones visit the remaining $\alpha \nrdrones$ nodes, both vehicles' tours would need exactly time $2$.

	Consider a deterministic online policy \policy.
	Consider the situation at time $\theta:=\frac{2(\alpha-1)+1}{\alpha}$. At this point in time, each drone has  visited at most $\alpha$ nodes -- and if it has visited $\alpha$ nodes, it has just arrived at the last one of these.
	Each truck has visited either $0$ or $1$ nodes -- denote the respective set of trucks by $T_0$ and $T_1$.

	We denote by $C^*$ the set nodes that are either not yet visited at time $\theta$, or where a drone has arrived at time $\theta$. Note that $|C^*|\ge \nrdrones +|T_0|$.

	Note that since the distance between any two nodes in $C^*$ is two and the distance between any node in $C$ to the depot is one, any truck in set $|T_0|$ has distance $\ge 1$ to all or all but one node in $C^*$. Thus, there are at least $\nrdrones$ nodes to which all trucks in $T_0$ have a distance of $\ge 1$.

	Assume that one of them is damaged.

	Then it takes each truck in $T_0$ time $\ge 1$ to reach it, and for each truck in $T_1$, the time to get there is $\ge 2-(\theta-1).$

	Thus, when taking a truck from $T_0$, the makespan is $\theta+2$ as the truck has to return to the depot afterwards, for a truck from set $T_1$ it is $\theta+2-(\theta-1) +1=4$.

	This shows that the competitive ratio  is
	$\frac{\ALG(I)}{\OPT(I^*)}    \ge \frac{2+\theta}{2}    = \frac{2+\frac{2(\alpha-1)+1}{\alpha}}{2}     = 2-\frac{1}{2\alpha}$.
\end{proof}

\lemalphagenlowernew*
\begin{proof} \label{proof:lem-alpha1-genlower-new}
	Note that if $\alpha\ge 2$ or $\alpha\le \frac{1}{2}$, $\sigma\ge 1$ and there is nothing to show.
	For all other cases, consider an instance $I$  with $|D|\leq k^t$ and with the depot $\depot$ and non-depot nodes $C$ forming a
	star graph $\gstar(\nrtrucks+\nrdrones,1)$.
	\begin{enumerate}
		\item If $\alpha\in [1,2]$, we have $\min\{\frac{2}{\alpha},2\alpha\}=\frac{2}{\alpha}$.
		      The full-information optimum has a makespan of $2$, because we can visit each damaged node with a truck and use the remaining vehicles to visit the other nodes (in time $\frac{2}{\alpha}\le 2$) to obtain a solution with this makespan. \\
		      Now assume that we have a policy with competitive ratio $<\frac{2}{\alpha}$.
		      Note that in such a policy, no vehicle can visit more than one node, because that would lead to a makespan of at least $\frac{4}{\alpha}$ and thus a competitive ratio of $\frac{\frac{4}{\alpha}}{2}=\frac{2}{\alpha}$.
		      Thus, each vehicle visits exactly one node.  Now assume that one of the nodes visited by drone is damaged (and none of the other nodes is). So one of the trucks has to visit the 'drone node' as well, we have $\nrtrucks+\nrdrones+1$ visits in total to be executed with $\nrtrucks+\nrdrones$ vehicles, thus one vehicle executes two visits. Proven by contradiction.
		\item If $\alpha\in (\frac{1}{2},1]$, we have $\min\{\frac{2}{\alpha},2\alpha\}=2\alpha$.
		      The full-information optimum has a makespan of $\frac{2}{\alpha}$, because we can visit each damaged node with a truck and use the remaining vehicles to visit the other nodes to obtain a solution with this makespan. Optimality of this solution can be seen because any solution where a vehicle visits two nodes has makespan $\ge 4 \ge \frac{2}{\alpha}$.\\
		      Assume that there is an online deterministic policy that achieves a competitive ratio $<2\alpha$. This means that the makespan of this policy is $< 2\alpha \frac{2}{\alpha}= 4$.
		      Note that in such a policy, no vehicle can visit more than one node, because that would lead to a makespan of at least $4$. Thus, each vehicle visits exactly one node.  Now assume that one of the nodes visited by drone is damaged (and none of the other nodes is). So one of the trucks has to visit the 'drone node' as well, we have $\nrtrucks+\nrdrones+1$ visits in total to be executed with $\nrtrucks+\nrdrones$ vehicles, thus one vehicle visits at least two nodes. Proven by contradiction.
	\end{enumerate}
\end{proof}

\section{Benchmark graphs} \label{sub:benchmarkgraphs}
We generate different benchmark graphs to evaluate our approaches experimentally.
The Random graph class is already sufficiently described in Section~\ref{sub:instances}.
Here we give a more detailed description of the remaining graph classes.

For the \emph{1-Center} and the \emph{2-Center} instances, we follow the approach described in
\cite{opt-approaches-for-tsp-with-drones}.
For the 1-center instances, we locate the depot at $(0, 0)$ and sample potential demand nodes by drawing an angle $a \in [0, \pi]$ uniformly,  and a distance $r$ from a normal distribution with mean $0$ and standard deviation $50$.  Then the position of the node is given by $(r \cos(a), r \sin(a))$.
For the $2$-center instances, we do the same but add another central node at $(200, 0)$ and translate a sampled node by $(200, 0)$ with a probability of $50 \%$, so that in the end the nodes are distributed around the two centers $(0, 0)$ and $(200, 0)$.
We construct a complete graph over these nodes, taking the Euclidean distance as edge labels.

The \emph{coastal} graphs are generated based on the idea that nodes represent cities in coastal regions. Typically, cities are mostly settled near the coastline, while cities further land inwards are more likely to settle perpendicular to cities on the coast.
The nodes of the coastal graphs are uniquely sampled according to a probability function, which is influenced by the previously sampled nodes, on a $101 \times 101$ lattice with dimensions $1 \times n / 15$, where $n$ is the number of sampled nodes.
The probability is affected by its distance to the left side of the lattice (which represents the coastline) and the previously sampled nodes:
Nodes around sampled nodes have a probability of zero, while nodes perpendicular to other cities with a relatively small land inwards $x$-coordinate differences have an increased probability.
Also for the coastal graphs, edge labels correspond to Euclidean distances.
However, coastal graphs are not complete.
The edges between two nodes are sampled in a two-phase procedure.
In the first phase, each edge is assigned a probability based on the Euclidean distance between the nodes.
Then we create an maximum spanning tree according to the edge probability with Kruskal's algorithm~\citeA{mst-kruskal}.
In the second phase, the probability is reduced if the angle between the considered edge and an already chosen edges is shallow and set to $0$ if it would cross an existing edge.
We add edges iteratively by sampling with this probability function over all unconnected node pairs until the desired number of graphs is reached.

The \emph{mountain} class of graphs represent networks of cities in mountain areas.
Cities are prominently built in the valleys between mountains, and less likely further up the mountains.
This model works analogously to the coastline mode.
Here, the lattice has dimensions $\sqrt{n /15} \times \sqrt{n / 15}$ and the probability function is derived from the height of the node on a mountain, where higher nodes have a lower probability.
Also, we set the probability of nodes near other accepted nodes to zero.
In the first phase of the edge construction, the probability of an edge scales with the inverse of the height of the highest point in the straight line between the incident nodes on any mountain and the Euclidean distance between the two incident nodes.
In the second phase, we iteratively sample over the unused edges, where in addition shallow angles diminish the probability of edges and crossing set the probability of edges to $0$.

A more detailed explanation of the graph generation process and the standalone implementation can be found in \citeA{neugebauer2026gengraphAppendix}.
The used benchmark graphs and a standalone instance generator are provided in \citeA{neugebauer2026implementationAppendix}.

\section{Implementation of the policies }\label{sec:implementation}
All policies considered in this paper, as well as the $\truckonly$ policy used for comparison, require (repeated) solution of \tspmnname. In Sections~\ref{sub:mip} and~\ref{sub:dp}  we provide a  mixed-integer linear program (MILP) and a dynamic program (DP), respectively, to solve \tspmnname. We also comment on how these are modified to solve RDP$^*$ which we need to solve in order to compute competitive ratios.

On \citeA{neugebauer2026implementationAppendix} we provide a Rust~\citeA{rust} implementations of the policies.
For each policy we provide two implementations: one that solves \tspmnname\ by solving the MILP (using Gurobi $13.0.1$ \citeA{gurobi}) and one, only applicable for instances with $\nrdrones=\nrtrucks=1$, that uses the DP to solve \tspmnname.
For solving the \ISMALL\ instances we use the MILP-based implementation, for the other two instance-classes we use the DP-based one due to superior performance.

\subsection{Mixed-Integer Linear Programming Formulation}
\label{sub:mip}

Consider an instance $(G,\alpha,m,n,W)$ of \tspmnname\ where
$G=(V,E,c)$ is an undirected graph with a designated depot $\depot\in V$ and positive edge labels $c$, $\alpha$ is the drone speed, $m$ and $n$ with $m+n\geq 1$ denote the numbers of trucks and drones, respectively, and a node set $W\subseteq V\setminus \{\depot\}$.
For the formulation of the MILP, we build a directed complete graph $G' = (V', E',c')$ with $V':=W\cup\{\depot\}$
by taking a complete directed graph over the vertices $V'$ and assigning each edge $(u, v) \in E'$ the weight $c'(u,v) > 0$ of the length of the shortest $u$--$v$ path in $G$.
To distinguish the different tours and their corresponding vehicles, we introduce the indices $I^t = \set{1, \dots, m}$ and $I^d = \set{m + 1, \dots, m + n}$ for the drone and truck tours respectively.
The set of all tour indices is given by $I = I^d \cup I^k$.
Let $\delta^-(v)$ be the set of edges in $G'$ originating from $v \in V'$ and $\delta^+(v)$ be the edges terminating in $v \in V'$.
The variable $x_e^i \in \set{0, 1}$ indicates if the edge $e \in E'$ is used in the tour of vehicle $i \in I$ and the variable $T \geq 0$ represents the final makespan, which should be minimized.
Finally, the variable $\theta_v \in [1, \abs{V}]$ indicates the position of node $v$ in the tour of the visiting vehicle.

\begin{mini!}
{}{T \nonumber}{}{}
\addConstraint{\sum_{e \in \delta^+(v)} x_e^i}{= \sum_{e \in \delta^-(v)} x_e^i}{\quad \forall v \in V' \; \forall i \in I {\label{milp:flow1}}}
\addConstraint{\sum_{e \in \delta^+(\depot)} x_e^i}{\leq 1}{\quad \forall i \in I \label{milp:flow2}}
\addConstraint{\sum_{i \in I} x_e^i}{\leq 1}{\quad \forall e \in E' \label{milp:disjoint_tours}}
\addConstraint{\sum_{e \in \delta^-(v)} \sum_{i \in I} x_e^i}{= 1}{\quad \forall v \in W \label{milp:vertex_visited}}
\addConstraint{T}{\geq \sum_{(u,v) = e \in E'} c(u, v) x_e^i}{\quad \forall i \in I^t \label{milp:lb_makespan_truck}}
\addConstraint{\alpha T}{\geq \sum_{(u, v) = e \in E'} c(u, v) x_e^i}{\quad \forall i \in I^d \label{milp:lb_makespan_drone}}
\addConstraint{\theta_v}{\geq \theta_u + 1 - \abs{V} (1 - x_{uv}^i)}{\quad \forall i \in I \; \forall (u,v) \in E', u, v \neq \depot \label{milp:mtz_sec}}
\addConstraint{x_e^i}{\in \set{0,1}}{\quad \forall e \in E' \; \forall i \in I \nonumber}
\addConstraint{T}{\in \mathbb{R}_{\geq 0} \nonumber}
\addConstraint{\theta_v}{\in [1, \abs{V}]}{\quad \forall v \in V' \setminus \set{\depot} \nonumber}
\end{mini!}

Constraints~\labelcref{milp:flow1} are flow constraints, they ensure that if a vehicle enters a node, it also leaves the node.
Constraints~\labelcref{milp:flow2} ensure that each vehicle makes only one tour.
Note that these constraints allow a vehicle to stay at the depot.
To strengthen the model, we added the constraint~\labelcref{milp:disjoint_tours} that enforces the vehicles' tours to be edge-disjoint.
The constraints~\labelcref{milp:vertex_visited} enforce that each node is visited by exactly one vehicle.
The makespan of each vehicle's tour gives a lower bound for the total makespan $T$, which is described via the constraints~\labelcref{milp:lb_makespan_truck,milp:lb_makespan_drone}.
Constraints~\labelcref{milp:mtz_sec} are Miller-Tucker-Zemlin subtour elimination constraints~\citeA{mtz-subtour-elimination-constraints}, which ensure that each node $v$ is assigned a label $\theta_v$ which corresponds to the position in the tour of the visiting vehicle.

To solve RDP$^*$, we set $W=C$ and add the constraint~\eqref{milp:trucks_visit_demand}  to enforce that all demand nodes $D$ are visited by the trucks.
\begin{equation}
	\label{milp:trucks_visit_demand}
	\sum_{i \in I^t} \sum_{e \in \delta^-(d)} x_e^i \geq 1 \quad \forall d \in D
\end{equation}

\subsection{Dynamic program}
\label{sub:dp}
Consider an instance $(G,\alpha,m,n,W)$ of \tspmnname\ as before.
Analogously to the preprocessing for the MILP, we build a (\emph{undirected}) complete graph
$G' = (V', E',c')$ with $V':=W\cup\{\depot\}$
by taking a complete (undirected) graph over the vertices $V'$ and assigning each edge $\{u, v\} \in E'$ the weight $c'(u,v) > 0$ of the length of the shortest $u$--$v$ path in $G$.

In the first step, we use the  Bellman-Held-Karp algorithm~\citeA{bellman-held-karp-1,bellman-held-karp-2} for the $\tsp$ to obtain (via dynamic programming, in time $\Oh(\abs{V}^2 2^{\abs{V}})$) a table $\bhktabelle[U, u]$, which, for all subsets $U \subseteq W$ and nodes $u \in U$, contains a shortest path starting from $\depot$, traversing all nodes in $U$, and ending in $u$.
Denote by $\ell(\bhktabelle[U,u])$ the length of such a shortest path.
After this step, we compute $\multitsphome_{m, n}(C)$ by iterating over all possible partitions of $W$ into $m + n$ disjoint subsets, which induce truck and drone tours, and taking the tours with the smallest makespan.
See Algorithm~\ref{alg:multi-tsp}.

Algorithm~\ref{alg:multi-tsp} has a runtime of $\Oh(\abs{V} (m + n)^{\abs{V}})$ due to the runtime required to iterate through all pairwise disjoint $m + n$ subsets of $C$, which dominates the runtime of Bellmann-Held-Karp's algorithm as long $\nrtrucks + \nrdrones \geq 2$.
The space consumption of Algorithm~\ref{alg:multi-tsp} is $\Oh(\abs{V} 2^{\abs{V}})$ with a dynamic representation of the paths due to the table size of $\bhktabelle$.

To solve RDP$^*$, we modify the algorithm such that in the \textbf{for}-loop only partitions with  $\bigcap_{i=1}^m U_i\supset D$ are considered.

We provide an implementation of the Algorithm~\ref{alg:multi-tsp} and the MILP of Section~\ref{sub:mip} in \citeA{neugebauer2026implementationAppendix}.

\begin{algorithm}
	\caption{Dynamic programm for \tspmnname}
	\label{alg:multi-tsp}
	\begin{algorithmic}[1]
		\Require Graph $G' = (V', E', c')$ with edge-weights $c > 0$, $\alpha > 0$, number of trucks and drones $m, n \in \N_0$.
		\Ensure Up to $m$ truck and $n$ drone tours in $G'$ that cover all nodes and minimize $\tsp_{m,n}(W)$.
		\State Compute table $\bhktabelle$ with Bellmann-Held-Karp's algorithm
		\State $\ttours = \set{\emptyset \colon i = 1, \dots, m}$
		\State $\dtours = \set{\emptyset \colon i = 1, \dots, n}$
		\State $T = \infty$
		\For{pairwise disjoint subsets $U_1, \dots, U_{m + n} \subseteq C$ with $\cup_{i = 1}^{m + n} U_i = W$}
		\State $u_i = \argmin{u \in U_i } \ell(\bhktabelle[U_i, u]) + c(u, \depot)$
		\State $T' = \max_{i = 1, \dots, m + b} \beta_i \left( \ell(\bhktabelle[U_i, u_i]) + c(u_i, \depot) \right)$ where $\beta_i = 1$ if $i \in \set{1, \dots, m}$ and $\beta_i = \alpha$ else.
		\If{$T' < T$}
		\State $T = T'$
		\State $\ttours = \set{\bhktabelle[U_i, u_i] \cup \set{c(u_i, \depot)} \colon i = 1, \dots, m}$
		\State $\dtours = \set{\bhktabelle[U_i, u_i] \cup \set{c(u_i, \depot)} \colon i = m+1, \dots, m + n}$
		\EndIf
		\EndFor
		\State \Return $\ttours, \dtours$
	\end{algorithmic}
\end{algorithm}

\bibliographyA{mybib}

\end{document}